\documentclass[12pt,twoside,a4paper]{article}
\usepackage{amsmath}
\usepackage{amssymb}
\usepackage{graphicx}
\usepackage{syntonly}
\usepackage{amsthm}
\usepackage{newclude}
\usepackage{stmaryrd}
\usepackage[margin=1.2in]{geometry}
\usepackage{imakeidx}
\usepackage[nottoc,numbib]{tocbibind}

\usepackage{fancyhdr}
\usepackage[draft]{optional}
\usepackage[british,american]{babel}
\usepackage{lipsum}
\usepackage{mathrsfs}
\usepackage{amsfonts, wasysym, caption}
\usepackage{paralist}
\usepackage{comment}
\usepackage{amsopn}
\usepackage{accents,bbm,bibgerm,eucal,paralist,url,verbatim,wasysym}
\usepackage[normalem]{ulem}
\usepackage{
	latexsym,
	nicefrac
}
\usepackage[usenames]{color}
\usepackage{tikz}
\usetikzlibrary{decorations.pathreplacing}
\usepackage{url}
\usepackage[multiple]{footmisc}

\definecolor{darkgreen}{rgb}{0,0.5,0}
\definecolor{darkred}{rgb}{0.7,0,0}
\usepackage[colorlinks, 
citecolor=darkgreen, linkcolor=darkred
]{hyperref}

\DeclareMathAlphabet{\mathcal}{OMS}{cmsy}{m}{n}

\newtheorem{theorem}{Theorem}[section]
\newtheorem*{theorem*}{Theorem}
\newtheorem{lemma}[theorem]{Lemma}
\newtheorem{proposition}[theorem]{Proposition}
\newtheorem{definition}[theorem]{Definition}

\newcommand\numberthis{\addtocounter{equation}{1}\tag{\theequation}}

\makeatletter
\newcommand\RedeclareMathOperator{%
	\@ifstar{\def\rmo@s{m}\rmo@redeclare}{\def\rmo@s{o}\rmo@redeclare}%
}
\newcommand\rmo@redeclare[2]{%
	\begingroup \escapechar\m@ne\xdef\@gtempa{{\string#1}}\endgroup
	\expandafter\@ifundefined\@gtempa
	{\@latex@error{\noexpand#1undefined}\@ehc}%
	\relax
	\expandafter\rmo@declmathop\rmo@s{#1}{#2}}
\newcommand\rmo@declmathop[3]{%
	\DeclareRobustCommand{#2}{\qopname\newmcodes@#1{#3}}%
}
\@onlypreamble\RedeclareMathOperator
\makeatother

\newcommand{\al}{\alpha}

\newcommand{\de}{\partial}

\newcommand{\lbar}{\overline}
\newcommand{\td}{\widetilde}
\newcommand{\bsl}{\backslash}
\renewcommand{\bf}{\mathbf}

\newcommand{\wlim}{\overset{w}{\rightharpoonup}}
\newcommand{\wslim}{\overset{w^*}{\rightharpoonup}}

\renewcommand{\L}{\mathcal L}
\newcommand{\ve}{\varepsilon}
\newcommand{\vf}{\varphi}
\newcommand{\oo}{\infty}
\newcommand{\ps}{\psi}
\newcommand{\et}{\eta}
\renewcommand{\th}{\theta}

\newcommand{\Ph}{\Phi}
\newcommand{\om}{\Omega}

\newcommand{\dt}{\delta}
\newcommand{\p}{\rho}

\newcommand{\B}{\mathcal B}

\newcommand{\M}{\ensuremath{{\mathcal M}}}

\newcommand{\R}{\Bbb R}

\DeclareMathOperator{\dist}{dist}

\DeclareMathOperator{\cof}{cof}

\RedeclareMathOperator{\div}{div}

\DeclareMathOperator{\sgn}{sgn}

\newcommand{\res}[1]{\left.#1\right\vert}

\newcommand{\Hd}[1]{\ensuremath{{\mathcal H^{#1}}}}

\newcommand{\W}[1]{\textup{W}^{#1}}
\newcommand{\Leb}[1]{\textup{L}^{#1}}

\numberwithin{equation}{section}

\def\Xint#1{\mathchoice
	{\XXint\displaystyle\textstyle{#1}}%
	{\XXint\textstyle\scriptstyle{#1}}%
	{\XXint\scriptstyle\scriptscriptstyle{#1}}%
	{\XXint\scriptscriptstyle\scriptscriptstyle{#1}}%
	\!\int}
\def\XXint#1#2#3{{\setbox0=\hbox{$#1{#2#3}{\int}$ }
		\vcenter{\hbox{$#2#3$ }}\kern-.6\wd0}}

\def\dashint{\Xint-}

\title{\vspace{-2cm}The positivity of the Jacobian in the weak limit of generalised axisymmetric maps}
\author{Duvan Henao$^{1,2}$, \and Panas Kalayanamit$^{1}$}
\date{%
\footnotesize{$^1$Instituto de Ciencias de la Ingenier\'ia, Universidad de O'Higgins,\\ Av.\ Bernardo O'Higgins 611, Rancagua, Chile.} \\ %
\footnotesize{$^2$Center for Mathematical Modeling, FB210005, Basal ANID Chile.} %
}

\begin{document}
	
	\maketitle
	
\section{Introduction}	
	
 \paragraph{}	
 In the calculus of variations, one often encounters situations where the weak limit $u$ of a sequence of maps $(u_j)_j$ loses some desirable properties that each $u_j$ has, which leads to complications such as the lack of compactness for the admissible space to a variational problem. An important example is that weak $\W{1,p}$ limit does not preserve the positivity of the Jacobians in general. The objective of this work is to show that the techniques Hencl and Onninen \cite{HenclOnninen2018} developed for sequences of Sobolev homeomorphisms can be adapted to show the positivity of the Jacobians in the weak limit of sequences of maps $u_j \in \W{1,2}(\om;\R^3)$ that are not necessarily continuous, provided that some conditions are satisfied. In particular, we show that generalised axisymmetric maps satisfy these conditions, hence the weak limit $u$ of a sequence $(u_j)_j$ of generalised axisymmetric maps satisfies $\det Du \ge 0$ a.e., provided that $\det Du_j > 0$ a.e.
 
 The positivity of the Jacobian of a map $u\colon\om \to \R^3$ is a natural assumption in nonlinear elasticity, where $u$ represents a physical deformation of an elastic body whose reference configuration is represented by an open domain $\om \subset \R^3$ \cite{Ball1977, MarsdenHughes1983, Ciarlet1988, Sverak1988}. The energy of a deformation $u$ is given by some polyconvex energy, e.g.\ the neo-Hookean energy
 \[
 E(u) = \int_\om |Du(\bf x)|^2 + H(\det Du(\bf x)) \,d\bf x,
 \]
 where $H(J) \to +\oo$ as $J\searrow 0$ and $H(J) = +\oo$ for $J<0$. The set of points $\{ \bf x \in \om: \det Du(\bf x) \le 0 \}$ thus corresponds to the part of the elastic body in which infinite compression or orientation-reversal occur, which are non-physical, hence only deformations in $\W{1,2}(\om;\R^3)$ with (strictly) positive Jacobian a.e. are admissible. One of the main challenges in finding a minimizer of $E$ using the Direct Method is the fact that, in general, a sequence $(u_j)_j$ in $\W{1,p}$, where each $u_j$ satisfies $\det D u_j > 0$ a.e., may converge weakly to some $u$ such that $\det Du < 0$ on a set of positive measure. Indeed, the construction in \cite[Chapter 8.5]{IwaniecMartin2001} shows that there exists a sequence of $u_j\in \W{1,2}(\om;\R^3)$ such that $\det Du_j > 0$ a.e., but $u_j \wlim u$ where $\det Du = -1$ on the whole $\om$. The same phenomenon persists even when each $u_j$ is a continuous map (see the construction in \cite{Maly1995}). However, an important result by Hencl and Onninen \cite{HenclOnninen2018} shows that the sign of the Jacobian cannot reverse in the weak limit when each $u_j$ is a \emph{homeomorphism} in $\W{1,2}(\om;\R^3)$ with $\det Du_j > 0$ a.e.\ (the limit map $u$ need not be a homeomorphism). Their result means that the weak closure of the class of Sobolev homeomorphisms with positive Jacobian (provided that some extra conditions are satisfied) is a good function space to find a minimizer of $E$ on (as was done in \cite{DolezalovaHenclMolochanova2023} and \cite{Kalayanamit2024}). 
 
 Aside from the compactness issue discussed above, the sign-reversal of the Jacobian in the limit is also an obstruction to proving the weak convergence of $(\det Du_j)_j$, and hence to proving the lower semicontinuity of the energy functional $E$ \cite{BallCurrieOlver1981}. In many cases, under some standard assumptions in nonlinear elasticity, one can show that  
 \[
 \det Du_j \wlim |\det Du|  \quad\text{in }\Leb{1}(\om;\R^{3\times 3}),
 \]
 when $u_j \wlim u$ in $\W{1,p}$, but removing the absolute signs to get the full weak convergence result for the Jacobians is often very difficult unless one knows that $\det Du \ge 0$ a.e. Examples of results of the above type are \cite[Theorem 4.1]{MullerSpector1995} and \cite[Theorem 2]{HenaoMora-Corral2010}. In \cite{MullerSpector1995}, one may deduce $\det Du_j \wlim \det Du$ in $\Leb{1}$ by further assuming that $p>2$ and that condition (INV) holds  (see \cite[Theorem 4.2]{MullerSpector1995}), whereas in \cite{HenaoMora-Corral2010} the full weak convergence of the Jacobians was achieved by assuming that $\cof Du_j \wlim \cof Du$ in $\Leb{1}$ and that the divergence identities \eqref{div_identity} hold (or a less restrictive condition that the sequence $(u_j)_j$ has uniformly bounded surface energy, see \cite[Theorem 3]{HenaoMora-Corral2010}). There are other results in the literature that guarantee the full convergence of the Jacobian, such as a very general result by Fonseca--Leoni--Mal\'y \cite{FonsecaLeoniMaly2005}, but one needs to assume extra integrability on the cofactors, i.e.\ that $\cof Du_j \in \Leb{\frac{3}{2}}$, which does not necessarily hold for $u_j \in \W{1,2}(\om;\R^3)$. 
 
 In this work, we expand on the idea of Hencl and Onninen in \cite{HenclOnninen2018} by showing that much of the techniques used in their paper, originally used for sequences of Sobolev homeomorphisms, can in fact be adapted to a more general type of maps. The essential ingredient in their work is the ingenious usage of the \emph{linking number}, which is inherently a topological notion, to analyse the weak convergence of Sobolev maps (in a prior work \cite{HenclMaly2010}, the linking number was used to show that sense-preserving Sobolev homeomorphisms have positive Jacobian a.e.). For other recent works that use the linking number to analyse properties of Sobolev maps, see \cite{GoldsteinHajlasz2019} and \cite{BouchalaHenclZhu2024}. 
 
 This paper shows that the analysis of the linking number leading to the positivity of $\det Du$ in \cite{HenclOnninen2018} can be extended to spaces of Sobolev maps that are not homeomorphisms. Specifically, we prove that:  
 
 \begin{theorem}\label{thm_main}
 	Let $\om\subset \R^3$ be a bounded domain, $(u_j)_j$ be a sequence of $\W{1,2}(\om;\R^3)$ generalised axisymmetric maps that are one-to-one a.e.\ and satisfy $\det Du_j > 0$ a.e. Suppose that $u_j \wlim u$ for some $u\in\W{1,2}(\om;\R^3)$, then $\det Du \ge 0$ a.e.
 \end{theorem}
 
 As in \cite{HenclOnninen2018}, for the proof we need to show that the linking number of the canonical link is preserved under the composition with Sobolev maps in our class. Instead of relying on \cite[Proposition 4.1]{HenclMaly2010}, which does not apply in our case, we base ourselves on the fact that orientation-preserving generalised axisymmetric $\W{1,2}$ maps locally satisfy the divergence identities \eqref{div_identity} restricted to $2$-dimensional planes. The novelty of our work is recognising that linking number is useful in analysing Sobolev maps even for those that are not homeomorphisms, as all of the previous work that uses the linking number in this direction has that requirement.
 
 This paper is a steppingstone to a more general result that we wish to prove, which is removing the generalised axisymmetric assumption from the proof, and show that the divergence identities \eqref{div_identity} on the full $3$-dimensional space is enough to prove the positivity of $\det Du$. This more general statement, once proved, would significantly strengthen recent works on the minimization of the neo-Hookean energy  \cite{BarchiesiHenaoMora-CorralRodiac2023_relaxation} and \cite{Kalayanamit2024}, where the formation of the pathological harmonic dipoles \cite{ContiDeLellis2003, BarchiesiHenaoMora-CorralRodiac2023_harmonic, BarchiesiHenaoMora-CorralRodiac2024, DolezalovaHenclMaly2023} must be ruled out.

\section{Preliminaries}

 \paragraph{Notations:}
 For $n\in \Bbb N$, $\bf x\in \R^n$ and $r>0$, we let $B_r^n(\bf x)$ denotes the open ball in $\R^n$ of radius $r$ centred at $\bf x$. $B_r^n$ denotes the ball of radius $r$ centred at $0$, whereas $B^n$ denotes the unit ball in $\R^n$. The superscript $n$	will usually be dropped if $n = 3$. $\om$ will usually denote a bounded domain in $\R^3$, whereas $\Lambda$ will usually denote a bounded domain in $\R^2$.

 \paragraph{}
 For each point $\bf x = (x_1, x_2, x_3) \in \R^3\bsl\{ (0, 0, t) : t \in \R \}$, there exist a unique $r > 0$ and a unique $\th \in [0, 2\pi)$ such that $x_1 = r\cos\th$ and $x_2 = r\sin\th$. We shall write it as
 \[
 (x_1, x_2, x_3) = (r, \th, x_3)_{cyl}
 \]
 and call this representation the cylindrical coordinates of $\bf x$. We shall denote the (open) half-plane in $\R^3$ of angle $\th$ in the cylindrical coordinates by
 \[
 O_\th := \{ (r, \th, z)_{cyl} : r>0, z \in \R \}
 \]
 and the plane whose $3^{\text{rd}}$ coordinate is $z$ by
 \[
 H_z := \{ (x,y,z) : x,y \in \R \}.
 \]
 
 \begin{definition}\label{defn_axisym}
 	Let $\om \subset \R^3$ be an open domain. A map $u\colon\om \to \R^3$ is said to be a \emph{generalised axisymmetric map} if there exist a strictly increasing, absolutely continuous function $\Theta\colon[0,2\pi] \to \R$ such that $\Theta(2\pi) = \Theta(0) + 2\pi$ and $\Theta' > 0$ a.e., $\td u_1 \colon[0,\oo)\times [0,2\pi) \times \R \to [0,\oo)$ and $\td u_2 \colon[0,\oo)\times [0,2\pi) \times \R \to \R$ such that
 	\[
 	u(r\cos \th, r\sin \th, z) = \td u_1(r, \th, z)\left( \cos(\Theta(\th))\bf{e_1} + \sin(\Theta(\th))\bf{e_2}  \right)  + \td u_2(r, \th, z)\bf{e_3}
 	\] 
 	in the standard coordinates, or, equivalently ,
 	\[
 	u \colon (r,\th,z)_{cyl} \mapsto (\td u_1(r,\th,z), \Theta(\th), \td u_2(r,\th,z))_{cyl}
 	\]
 	in the cylindrical coordinates with $\de_\th \Theta >0$ a.e. This means that $u$ maps $O_\th$ into $O_{\Theta(\th)}$. In particular, when $\Theta(\th) = \th$, and $\td u_1$, $\td u_2$ do not depend on $\th$, $u$ is called an \emph{axisymmetric map}. 
 \end{definition}

 \paragraph{}
 We shall refer the readers to \cite{GiaquintaModicaSoucek1998_book}, \cite{Federer1996_GMT}, \cite{EvansGariepy1992_1st} or \cite{MullerSpector1995} for the definition and properties of the approximate derivative of a measurable function $w$ defined on (a subset of) $\R^n$. In particular, we shall use that if $w$ is approximately differentiable at $x_0$, then $w$ is defined and is approximately continuous at $x_0$. It is a well-known fact that a Sobolev function $w$ is approximately differentiable a.e. and its approximate differential coincides a.e.\ with its distributional derivative $Dw$. The set of approximate differentiability of a Sobolev function plays an important role in nonlinear elasticity.
 
 \begin{definition}\label{defn_approx_diff_set}
 	Let $\Lambda \subset \R^n$ be an open set. For a function $w \in \W{1,1}(\Lambda;\R^n)$, the set of approximate differentiability of $w$ is denoted by
 	\[
 	\Lambda_d := \{\bf x \in \Lambda:  w \textup{ is approximately differentiable at $\bf x$} \}
 	\]
 	and is a set of full measure in $\Lambda$. 
 	
 	Moreover, it is known\footnote{See \cite[Definition 3, Lemma 3]{HenaoMora-Corral2011} and also \cite[Lemma 3.4]{MullerSpector1995}.} that there exists a set $\Lambda_0 \subset \Lambda_d$ such that $\Lambda_d\bsl\Lambda_0$ is a null set and $\res{w}_{\Lambda_0}$ is one-to-one whenever $w$ is one-to-one a.e.\ and $\det Dw >0$ a.e.
 \end{definition}

 An important property of $\Lambda_d$ is that whenever $w \in \W{1,1}(\Lambda;\R^n)$ with $\det Dw > 0$ a.e. and $w$ is one-to-one a.e., we have  the \textit{change of variables formula}
 \[
 \int_E (\vf\circ w)(\bf x) \det Dw(\bf x) \,d\bf x = \int_{w(E\cap \Lambda_d)} \vf(\bf y) \,d\bf y
 \]
 for any measurable $E\subset \Lambda$ and any measurable function $\vf\colon \R^n \to \R$ (see \cite[Proposition 2.6]{MullerSpector1995}). In particular, this implies that $w(E\cap \Lambda_d)$ is null whenever $E$ is null. Note that we can replace $\Lambda_d$ in the above formula with $\Lambda_0$.

 \paragraph{}
 In this work, we shall use the following planar version of the divergence identities.
 \begin{definition}
 	Let $\Lambda \subset \R^2$ be an open domain. A Sobolev function $w\in \W{1,2}(\Lambda;\R^2)$ is said to satisfy the \textit{divergence identities} if
 	\begin{equation}\label{div_identity}
 		\int_{\Lambda} (\div g)(w(\bf x)) \phi(\bf x) \det Dw(\bf x) \,d\bf x = - \int_{\Lambda} g(w(\bf x)) \cdot (\cof Dw(\bf x))[D\phi(\bf x)] \,d\bf x
 	\end{equation}
 	for every $\phi\in C^1_c(\Lambda)$ and every $g\in C^1_c(\R^2;\R^2)$.
 \end{definition}
 These identities and their applications in nonlinear elasticity can be traced back to \cite{ GiaquintaGiuseppeSoucek1989,Mueller1988}.

 \paragraph{}
 In proving that the weak $\W{1,p}$ limit of orientation-preserving homeomorphisms has non-negative Jacobian a.e., Hencl and Onninen \cite{HenclOnninen2018} rely on the fact that these maps do not reverse the signs of the \emph{linking number} of every link in $\om$.\footnote{A link is an ordered pair of closed curves (see Section 4). Since the linking number is a topological invariant, a sense-preserving homeomorphism does not change the sign of the linking number of any link (see \cite[Proposition 4.1]{HenclMaly2010} or \cite[Proposition 51]{GoldsteinHajlasz2019} for the proof). } Here, we show that any generalised axisymmetric map that has positive Jacobian a.e.\ and is one-to-one a.e.\ satisfies a similar property, namely, it preserves the sign of a sufficiently large number of links in $\om$ (from Section 3 onwards we shall refer to this property as \emph{property \hyperref[defn_property_L]{(L)}}). Even though  the correct sign may not be obtained for every link, this is shown (in Theorem \ref{thm_main_for_property_L}) to be strong enough to carry out the argument in  \cite{HenclOnninen2018} and obtain that weak limits of Sobolev maps satisfying \hyperref[defn_property_L]{(L)} have non-negative a.e.\ Jacobian. That yields Theorem \ref{thm_main}.
  
 We now describe briefly the content of the rest of this paper. Section 3 is devoted to establishing estimates involving a family of $C^1$ maps that parametrize the balls $B_{4r}(\bf x_0)$ in $\om$. These estimates will be later used in Theorem \ref{thm_main_for_property_L} when we integrate on links in $B_{4r}(\bf x_0)$. Note that these parametrizations correspond to the affine transformations $\bf z \mapsto \bf x_0 + r \bf z$ that are implicitly used in \cite{HenclOnninen2018}. In this work, we require more complicated parametrizations since we wish to align our links so that one of the curves in each link lies in some half-plane $O_\th$.
 
 In Section 4, following the approach in \cite{HenclOnninen2018} we define the linking number $\L(\vf,\ps)$ of an arbitrary link $(\vf,\ps)$ and introduce a family of links $(\mu_{\bf a}, \nu_{\bf b} )$ that are para\-metrized by $\bf a\in B^{2}_{1/10}(\bf 0)$ and $\bf b \in B^{2}(\bf 0) \cap B^{2}_{1/10}(\bf e_1)$. These links satisfy $\L(\mu_{\bf a}, \nu_{\bf b}) = 1$ for every $\bf a$ and $\bf b$. We then give a precise definition of property \hyperref[defn_property_L]{(L)} at the end of this section. 
 
 In Section 5, we show how $\L(v \circ \mu_{\bf a}, v \circ \nu_{\bf b})$, the linking number of the link $(\mu_{\bf a}, \nu_{\bf b})$ under a mapping $v$, can be calculated. In particular, the parametrizations that are introduced in Section 3 and the generalised axisymmetric assumption allow us to establish a formula relating $\L(v \circ \mu_{\bf a}, v \circ \nu_{\bf b})$ to the $2$-dimensional topological degree (or the winding number) of an associated planar map (Lemma \ref{lemma_linking_number_using_2D_degree}). To this end, we make use of the \emph{intersection number}, a tool from differential topology, in deriving the formula. Lastly, we make use of a result from \cite{BarchiesiHenaoMora-Corral2017} that allow us to compute the degree of an orientation-preserving Sobolev map satisfying the divergence identities \eqref{div_identity}.
 
 Section 6 is where the main theorems are proved. We end this paper with a short appendix as Section 7.

\section{Parametrizations and related estimates}

 Recall that for a function $f\in C^1(U;\R^n)$ for $U \subset \R^n$, we have a simple estimate 
 \begin{align*}
 	\left| f(\bf v) - f(\bf v_0) - Df(\bf v_0)[\bf v] \right| &\le |\bf v| \sup_{t\in [0,1]} \left| Df(t\bf v) - Df(\bf v_0) \right| \\
 	 &\le |\bf v| \sup_{\bf w \in U} \left| Df(\bf w) - Df(\bf v_0) \right|.
 \end{align*}
 This will be used in the following lemma.

 \begin{lemma}\label{lemma_local_C^1_parametrization}
 	Let $L \in C^1(\om';\R^n)$, where $\om'$ is an open neighbourhood of $\bf x_0 \in \R^n$, be a diffeomorphism such that $L(\bf x_0) = 0$ and $\det DL(\bf 0)>0$. Then, there exist $\al,R,R'>0$ (with $R'<R$) such that the family $\{ T_r \}_{r\in (0,R']}$, where each $T_r\colon U_r \to B_{4r}(\bf x_0)$ is a of $C^1$-diffeomorphism defined by
 	\[
 	T_r(\bf z) := L^{-1}\left( \frac{\al r \bf z}{R} \right),\quad\text{where}\quad \quad U_r := \frac{R}{\al r} L(B_{4r}(\bf x_0)),
 	\]
 	has the following properties:
 		\begin{itemize}
 		\item[(i)]  Each $U_r$ is an open set in $\R^n$ such that $B_4(\bf 0) \subset U_r$ and $T_r(\bf 0) = \bf x_0$.
 		\item[(ii)] There exists a constant $c=c(\bf x_0)>0$ such that $\det DT_r(\bf z) \ge c^{-1} r^n$ uniformly for all $\bf z \in B_{4r}(\bf x_0)$ and all $r \le R'$.
 		\item[(iii)] There exists an invertible linear map $\bf A \in \Bbb M^{n \times n}$ with $\det \bf A>0$ such that $\mathbf A = \frac{DT_r(\bf 0)}{r}$ for all $r \le R'$,
 		\[
 		\sup_{\bf z\in B_{4}(\bf 0)} \left|\bf A - \frac{DT_r(\bf z)}{r} \right| \le \frac{|\bf A|}{2} \quad\text{and}\quad\sup_{\bf z\in B_{4}(\bf 0)} \left|\bf A - \frac{DT_r(\bf z)}{r} \right| \to 0 \quad\text{as}\quad r\to 0.
 		\]
 	\end{itemize}
 \end{lemma}
 
 \begin{proof}
 	By the continuity of $DL(\bf x)$ on $\om'$, we may pick $R>0$ such that 
 	\begin{equation}
 		\frac{ \sup_{\bf x \in B_{4R}(\bf x_0) } | \det DL(\bf x_0) - \det DL(\bf x)| }{ \big( |DL^{-1}(\bf 0)| + \sup_{\bf y \in L(B_{4R}(\bf x_0))} | DL^{-1}(\bf y) - DL^{-1}(\bf 0) |  \big)^n } \le \frac{1}{2} \det DL^{-1}(\bf 0)
 	\end{equation}
 	and we then pick
 	\begin{equation}
 		\al := \frac{ R }{  |DL^{-1}(\bf 0)| + \sup_{\bf y \in L(B_{4R}(\bf x_0))} | DL^{-1}(\bf y) - DL^{-1}(\bf 0) |  }.
 	\end{equation}
 	
 	 To prove (i), we first observe that $B_4(\bf 0) \subset U_r$ is equivalent to $L^{-1}(B_{4\al r/R}(\bf 0)) \subset B_{4r}(\bf x_0)$. This follows from the calculation
 	 \begin{align*}
 	 	\left| L^{-1}\left(\frac{4\al r \bf z}{R} \right) - \bf x_0 \right| &\le \left|  L^{-1}\left(\frac{4\al r \bf z}{R} \right) - L^{-1}(\bf 0) - DL^{-1}(\bf 0)\left[ \frac{4\al r \bf z}{R} \right] \right| \\
 	 	&\quad + \left| DL^{-1}(\bf 0)\left[ \frac{4\al r \bf z}{R} \right] \right| \\
 	 	&< \frac{4\al r}{R} \sup_{\bf y \in L(B_{4r}(\bf x_0))} \left|DL^{-1}(\bf y) - DL^{-1}(\bf 0)  \right| + \frac{4\al r}{R}|DL^{-1}(\bf 0)| \\
 	 	&\le 4r
 	 \end{align*}
 	 for any vector $\bf z \in \R^n$  with $|\bf z|<1$. The fact that $U_r$ is open and $T_r(\bf 0)=\bf x_0$ for each $r\le R$ is obvious.
 	 
 	 To prove (ii), since $DT_r(\bf z) = \frac{\al r}{R}DL^{-1}(\frac{\al r\bf z}{R})$, we have 
 	 \begin{align*}
 	 	\det DT_r(\bf z) &= \left(\frac{\al r}{R} \right)^n \det DL^{-1}(\bf z) \\
 	 	&\ge \left(\frac{\al r}{R} \right)^n \left( \det DL^{-1}(\bf 0) - \left| \det DL^{-1}(\bf 0) - \det DL^{-1}(\bf z) \right| \right) \\
 	 	&= r^n \left( \frac{\det DL^{-1}(\bf 0) -  \left| \det DL^{-1}(\bf 0) - \det DL^{-1}(\bf z) \right| }{\big( |DL^{-1}(\bf 0)| + \sup_{\bf y \in L(B_{4R}(\bf x_0))} | DL^{-1}(\bf y) - DL^{-1}(\bf 0) |  \big)^n  } \right) \\
 	 	&\ge \left( \frac{\det DL^{-1}(\bf 0)}{2} \right)  r^n
 	 \end{align*}
 	 for all $\bf z \in B_{4}(\bf 0)$ and all $r \le R$. We may thus pick $c := 2 \det DL(\bf x_0)$.
 	 
 	 Lastly, to prove (iii), since $DT_r(\bf z) = \frac{\al r}{R}DL^{-1}(\frac{\al r\bf z}{R})$, it is not hard to see that $\frac{DT_r(\bf 0)}{r} = \frac{\al}{R} DL^{-1}(\bf 0)$ for all $r$, so we may let $\bf A := \frac{\al}{R} DL^{-1}(\bf 0)$. Therefore
 	 \begin{align*}
 	 	\sup_{\bf z\in B_{4}(\bf 0)} \left|\bf A - \frac{DT_r(\bf z)}{r} \right| &= \frac{\al}{R} \sup_{\bf z\in B_{4}(\bf 0)} \left| DL^{-1}(\bf 0) - DL^{-1}(\bf z) \right| \\
 	 	&\le \frac{\al}{R} \sup_{\bf x\in B_{4r}(\bf x_0)} \left| DL(\bf x_0)^{-1} - DL(\bf x)^{-1} \right|,
 	 \end{align*}
 	 which converges to $0$ as $r\to 0$ since $L\in C^1$. We then pick a smaller $R' \le R$ so that
 	 \[
 	 \sup_{\bf x\in B_{4R'}(\bf x_0)} \left| DL(\bf x_0)^{-1} - DL(\bf x)^{-1} \right| \le \frac{R|\bf A|}{2\al}.
 	 \] 
 	 This concludes the proof.
 \end{proof}

 \paragraph{}
 Properties (i), (ii) and (iii) in Lemma \ref{lemma_local_C^1_parametrization} of $\{ T_r \}_{r\in (0,R']}$ of $C^1$ allow us to get the following estimate, which is our counterpart of equation (8) in \cite{HenclOnninen2018}.

 \begin{lemma}\label{lemma_u_j_near_u_in_small_ball}
 	Let $(u_j)_j$ be a sequence in $\W{1,p}(\om;\R^n)$, $p \ge 1$, such that $u_j \wlim u$. Let $\dt > 0$ be given. For each $\bf x_0\in \om$, suppose that there exists a family $\{ T_r \}_{r\in (0,R']}$ of $C^1$ diffeomorphisms that satisfies properties (i), (ii) and (iii) in Lemma \ref{lemma_local_C^1_parametrization} (here $R'$ depends on $\bf x_0$). Then, for almost every $\bf x_0 \in \om$, there exists $R''<R'$ in which for every $r\in(0,R'')$, there is a corresponding $j_0 = j_0(\bf x_0,\dt,r) \in \Bbb N$ such that
 	\begin{equation}
 		\int_{B_4(\bf 0)} \left| \frac{ u_j(T_r(\bf z)) - u(\bf x_0) }{r} - Du(\bf x_0)[\bf A\bf z] \right| \,dz < \dt^3
 	\end{equation}
 	for all $j \ge j_0$.
 \end{lemma}
 
 \begin{proof}
 	For each $\bf x_0$, we let $c=c(\bf x_0)>0$ be the constant and $\bf A = \bf A(\bf x_0)$ be the linear map given by Lemma \ref{lemma_local_C^1_parametrization}. Recall that $u$ is approximately differentiable a.e.\ in $\om$, so we may let $\bf x_0 \in \om_d$ (see Definition \ref{defn_approx_diff_set}) so that
 	\[
 	\lim_{r\to 0} \dashint_{B_r(\bf x_0)} \left| \frac{u(\bf x) - u(\bf x_0) - Du(\bf x_0)[\bf x-\bf x_0]}{r} \right| \,d\bf x = 0.
 	\]
 	We can thus find $R_1 < R'$ such that 
 	\[
 	\frac{c}{r^n}\int_{B_{4r}(\bf x_0)} \left| \frac{u(\bf x) - u(\bf x_0) - Du(\bf x_0)[\bf x-\bf x_0]}{r} \right| \,d\bf x < \frac{\dt^3  }{3} \numberthis\label{eqn_u_j_near_u_in_small_ball_1}
 	\]
 	for every $r \in (0,R_1)$.
 	
 	From Lemma \ref{lemma_local_C^1_parametrization} (i), (ii) and the change of variable formula,
 	\begin{align*}
 		\int_{B_4(\bf 0)} f(T_r(\bf z)) \,d\bf z &\le \frac{c}{r^n} \int_{U_r} f(T_r(\bf z)) \det DT_r(\bf z) \,d\bf z \\
 		&= \frac{c}{r^n} \int_{B_{4r}(\bf x_0)} f(\bf x) \,d\bf x
 	\end{align*}
 	for any integrable function $f$, hence 
 	\begin{align*}
 		\int_{B_4(\bf 0)} &\left| \frac{ u_j(T_r(\bf z)) - u(\bf x_0) }{r} - Du(\bf x_0)[\bf A\bf z] \right| \,d\bf z \\
 		&\quad\quad\quad\quad \le \int_{B_{4}(\bf 0)} \left| \frac{ u_j(T_r(\bf z)) - u(\bf x_0) - Du(\bf x_0)[T_r(\bf z) - \bf x_0] }{r} \right| \,d\bf z + \frac{\dt^3}{3}  \\
 		&\quad\quad\quad\quad \le \frac{c}{r^n}\int_{B_{4r}(x_0)} \left| \frac{ u_j(\bf x) - u(\bf x_0) - Du(\bf x_0)[\bf x-\bf x_0]}{r} \right| \,d\bf x + \frac{\dt^3}{3} \numberthis\label{eqn_u_j_near_u_in_small_ball_2}
 	\end{align*}
 	whenever $r \in (0,R_2)$ for some $R_2 < R_1$. The first inequality follows from
 	\begin{align*}
 		\left| \mathbf A[\bf z] - \frac{T_r(\bf z)-\bf x_0}{r} \right| = \left| \frac{ T_r(\bf z) - T_r(\bf 0) - DT_r(\bf 0)[\bf z] }{r} \right| \le \frac{4}{r} \sup_{\bf z\in B_{4}(\bf 0)} | DT_r(\bf 0) - DT_r(\bf z) |,   
 	\end{align*}
 	which converges to $0$ as $r\to 0$ according to Lemma \ref{lemma_local_C^1_parametrization} (iii), so we may pick $R_2<R_1$ such that the last term is less than $\frac{\dt^3}{3|Du(\bf x_0)| |B_4(\bf 0)| }$ whenever $r < R_2$.
 	
 	Now, the weak convergence of $u_j$ to $u$ in $\W{1,p}(\om;\R^n)$ implies that $u_j \to u$ strongly in $\Leb{1}(\om;\R^n)$, hence for any $r \in (0,R_2)$, there exists $j_0=j_0(\bf x_0,\dt,r) \in \Bbb N$ such that
 	\[
 	\int_{B_{4r}(\bf x_0)} |u(\bf x) - u_j(\bf x) | \,dx <  r^{n+1} \frac{\dt^3}{3c} \numberthis\label{eqn_u_j_near_u_in_small_ball_3}
 	\]
 	for all $j > j_0$. Combining \eqref{eqn_u_j_near_u_in_small_ball_1}, \eqref{eqn_u_j_near_u_in_small_ball_2} and \eqref{eqn_u_j_near_u_in_small_ball_3} gives the result we want (with $R'' = R_2$).
 \end{proof}

 In the later sections, we shall need the map $L$ in Lemma \ref{lemma_local_C^1_parametrization} to have a specific form, so that the family $\{ T_r \}_{r\in(0,R']}$ generated from it parametrises the balls $B_{4r}(\bf x_0)$ in a way that $u\circ T_r$ has nice properties when $u$ is a generalised axisymmetric map. For that purpose, we define the following family of $C^1$ map on a neighbourhood of each $\bf x_0 \in \R^3$.
 
 \begin{definition}\label{defn_parametrization_L}
 	For any fixed $\bf x_0 \in \R^3\bsl\{ (0, 0, t) : t \in \R \}$, we let $r_0, \th_0, z_0$ be the numbers determined by $\bf x_0 = (r_0, \th_0, z_0)_{cyl}$. Consider the function $L_{\bf x_0}\colon B_{r_0}(\bf x_0) \to \R^3$ defined by
 	\begin{equation}\label{eqn_parametrization_L}
 		L_{\bf x_0}\colon (r, \th, z)_{cyl} \mapsto (r-r_0, -z + z_0, \th - \th_0).
 	\end{equation}
 	
 \end{definition}
 
 Note that the right hand side of \eqref{eqn_parametrization_L} is the standard coordinates in $\R^3$, while the left hand side is the cylindrical coordinates. We can easily see that $L_{\bf x_0}(\bf x_0) = \bf 0$, $\det DL_{\bf x_0}>0$ and that $L_{\bf x_0}$ maps subsets of $O_{\th}$ into the plane $H_{\th - \th_0}$.

\section{Links and the linking number}
 
 \paragraph{}
 Consider $\Phi\colon \lbar{B^{2}}\bsl\{\bf 0\} \times \lbar{B^{2}} \to \R^3$ defined by
 \[
 \Phi(\xi,\et) = \Phi(\xi_1,\xi_2,\et_1,\et_2) := \left( (|\xi|\et_1+2) \hat\xi_1, (|\xi|\et_1+2) \hat\xi_2, |\xi|\et_2 \right).
 \]
 Here we write $\hat\xi = \frac{\xi}{|\xi|}$ for $\xi \ne \bf 0$, i.e.\ $\hat\xi$ is the unit vector pointing in the direction of $\xi$.\footnote{Note that our definition of $\Phi$ is slightly different from the one in \cite{HenclMaly2010, HenclOnninen2018}.} We can see that $\Phi(S^1 \times B^{2})$ is the open anuloid 
 \[
  \Bbb A := \left\{  (x_1,x_2,x_3)\in \R^3 : \left( \sqrt{x_1^2 + x_2^2} - 2 \right)^2 + x_3^2 < 1 \right\}
 \]
 in $\R^3$, whereas $\Phi(S^1 \times S^1)$ is its surface $\de \Bbb A$.

 We define a \textit{link} to be a pair $(\vf, \ps)$ of (continuous) parametrized curves $\vf\colon S^1 \to \R^3$ and $\ps\colon S^1 \to \R^3$. The \textit{linking number} of the link  $(\vf, \ps)$ is defined using the topological degree
 \[
 \L(\vf,\ps) := \deg(L, \Bbb A, \bf 0) = \deg(L, \de \Bbb A, \bf 0),
 \]
 where $L = L_{\vf,\ps}\colon \lbar{\Bbb A} \to \R^n$ is defined as
 \[
 L(\bf x) = \vf(\xi(\bf x)) - \lbar\ps(-\et(\bf x)), 
 \] 
 provided that $L(\bf x) \ne \bf 0$ for all $\bf x\in \de \Bbb A$, or equivalently
 \[
 L(\Ph(\xi,\et)) = \vf(\xi) - \lbar\ps(-\et)\quad \text{for }\ \xi\in S^1, \et\in B^{2},
 \]
 where $\lbar\ps$ is an arbitrary continuous extension of $\ps$ from the circle $S^1$ to the whole closed ball $\lbar{B^{2}}$. How $\ps$ is extended does not matter since $\L_{\vf,\ps}$ only depends on the values of $L$ on the boundary $\de \Bbb A.$

 We shall define $\mu\colon S^1 \to \R^3$ and $\nu\colon S^1 \to \R^3$ by
 \begin{align*}
 	\mu(\xi) &:= \Phi(\xi,\bf 0), \\
 	\nu(\et) &:= \Phi(\bf e_1, \et).
 \end{align*}
 The pair $(\mu,\nu)$ is called the \textit{canonical pair}. The images of these maps are
 \begin{align}
 	\begin{split}
 		\mu(S^1) &= C_{hor} := \{ (x_1,x_2,0) \in \R^3 : x_1^2+x_2^2 = 4\}, \\
 		\nu(S^1) &= C_{ver} := \{ (x_1,0,x_3) \in \R^3 : (x_1-2)^2 + x_3^2 = 1 \}.
 	\end{split} \label{defn_circles_C_h_and_C_v}
 \end{align}

 \paragraph{}
 For each $\bf x \in \lbar{\Bbb A}$, there exists a unique pair $(\xi,\et) \in S^1 \times  \lbar{B^{2}}$ such that $\Phi(\xi,\et) = \bf x$. We shall denote these as $\xi(\bf x)$ and $\et(\bf x)$. In particular, for $\bf x = (x_1,x_2,x_3)$, direct computation shows that 
 \[
 \xi(\bf x) = \left( \frac{x_1}{\sqrt{x_1^2 + x_2^2}}, \frac{x_2}{\sqrt{x_1^2 + x_2^2}} \right), \et(\bf x) = \left( \sqrt{x_1^2 + x_2^2} - 2, x_3 \right), 
 \]
 thus 
 \[
 D\et(\bf x) = \begin{bmatrix}
 	\frac{x_1}{\sqrt{x_1^2 + x_2^2}} &\frac{x_2}{\sqrt{x_1^2 + x_2^2}} &0 \\
 	0 &0 &1
 \end{bmatrix}
 \]
 and hence $J_\et(\bf x) = \sqrt{x_1^2 + x_2^2}$ (here $J_\et = \sqrt{\det((D\et)(D\et)^t)}$).

 On the other hand, for each $\bf x \in \lbar{\Bbb A}\bsl C_{hor}$, there exists a unique pair $(\td\xi,\td\et) \in \lbar{B^{2}}\bsl\{\bf 0\} \times S^1$ such that $\Phi(\td\xi,\td\et) = \bf x$. We shall denote these as $\td\xi(\bf x)$ and $\td\et(\bf x)$. By computing in the cylindrical coordinates $(r,\th,z)$, we can see that
 \[
 \td\xi(\bf x) = \left( \sqrt{(r-2)^2 + z^2} \cos(\th), \sqrt{(r-2)^2 + z^2} \sin(\th) \right)
 \]
 thus 
 \[
 D_{(r,\th,z)}\td\xi = \begin{bmatrix}
 	\frac{(r-2)\cos(\th)}{\sqrt{(r-2)^2 + z^2}} &-\sqrt{(r-2)^2 + z^2}\sin(\th) &\frac{z\cos(\th)}{\sqrt{(r-2)^2 + z^2}} \\
 	\frac{(r-2)\sin(\th)}{\sqrt{(r-2)^2 + z^2}} &\sqrt{(r-2)^2 + z^2}\cos(\th) &\frac{z\sin(\th)}{\sqrt{(r-2)^2 + z^2}} 
 \end{bmatrix}.
 \]
 Recall that 
 \[
 \bf M := \frac{\de(r,\th,z)}{\de (x_1,x_2,x_3)} = \begin{bmatrix}
 	\cos(\th) &\sin(\th) &0 \\
 	-\frac{\sin(\th)}{r} &\frac{\cos(\th)}{r} &0 \\
 	0 &0 &1
 \end{bmatrix},\quad
 \bf M \bf M^t = \begin{bmatrix}
 	1 &0 &0 \\
 	0 &\frac{1}{r^2} &0 \\
 	0 &0 &1
 \end{bmatrix}
 \]
 and 
 \[
 D\td\xi = (D_{(r,\th,z)}\td\xi)\bf M ,
 \]
 hence, after some tedious calculation, we can see that 
 \[
 J_{\td{\xi}}(\bf x) = \frac{(r-2)^2+z^2}{r^2} = \frac{(\sqrt{x_1^2+x_2^2}-2)^2+x_3^2}{x_1^2+x_2^2}.
 \] 
 This shows that $J_{\td\xi}$ is bounded away from $0$ away from the circle $C_{hor}$.
 
 \paragraph{}
 Following \cite{HenclOnninen2018}, we now define the perturbed version of the curves in the canonical pair $(\mu,\nu)$.

 \begin{definition}\label{defn_mu_a_and_nu_b}
 	For $\bf a\in B^{2}_{1/10}(\bf 0)$ and $\bf b \in B^{2}(\bf 0) \cap B^{2}_{1/10}(\bf e_1)$, we define\footnote{Since we define $\Phi$ differently from \cite{HenclOnninen2018}, our $\nu_{\bf b}$ will also be slightly different. In particular, $\nu_{\bf b}(S^1)$ and  $\nu_{\bf b'}(S^1)$ are disjoint whenever $\bf b \ne \bf b'$. Our definition of $\mu_{\bf a}$, however, coincides with \cite{HenclOnninen2018}. }
 	\begin{align*}
 		\mu_{\bf a}(\xi) &:= \Ph(\xi,\bf a) \quad \text{for }\ \xi\in S^1, \\
 		\nu_{\bf b}(\et) &:= \Ph(\bf b,\et) \quad \text{for }\ \et\in S^1.
 	\end{align*}
 \end{definition}

 \paragraph{}
 It is not difficult to verify (e.g.\ via a direct computation or via homotopy) that $\L(\mu_{\bf a}, \nu_{\bf b}) = 1$ for every $\bf a$ and $\bf b$. We end this section by defining property \hyperref[defn_property_L]{(L)}.

 \begin{definition}[Property (L)]\label{defn_property_L}
 	 A function $u\colon \om \to \R^3$ is said to satisfy \emph{Property (L)} if, for almost every $\bf x_0 \in \om$, there exists a corresponding family $\{ T_r \}_{r\in(0,R']}$ generated by some $C^1$-diffeomorphism $L$ according to Lemma \ref{lemma_local_C^1_parametrization} such that for each $r \le R'$,  $v\colon B_4(\bf 0) \to \R^3$ defined by
 	 \begin{equation}\label{eqn_property_L_1}
 	 	 	v(\bf z) := u(T_r(\bf z)).
 	 \end{equation}
 	 satisfies the following conditions:
 	 \begin{itemize}
 	 	\item[(i)] There exist measurable sets $E_{hor} \subset B^{2}_{1/10}(\bf 0)$ and $E_{ver} \subset B^{2}_1(\bf 0) \cap B^{2}_{1/10}(\bf e_1)$ such that $E_{hor}$ is of full measure in $B^{2}_{1/10}(\bf 0)$, $E_{ver}$ is of full measure in $B^{2}_1(\bf 0) \cap B^{2}_{1/10}(\bf e_1)$, $v$ is continuous on $ \mu_{\bf a}(S^1)$ for every $\bf a \in E_{hor}$, and $v$ is continuous on $\nu_{\bf b}(S^1)$ for every $\bf b \in E_{ver}$.

 	 	\item[(ii)] There \textbf{does not exist} a pair of measurable sets $\td A \subset E_{hor}$ and $\td B \subset E_{ver}$ such that $\L^2(\td A) >0$, $\L^2(\td B) >0$, and for every $(\bf a,\bf b) \in \td A \times \td B$, the linking number of  $(v \circ \mu_{\bf a}, v \circ \nu_{\bf b})$ is well-defined and
 	 	\[
 	 		\L(v \circ \mu_{\bf a}, v \circ \nu_{\bf b}) < 0.
 	 	\]

 	 \end{itemize}
 \end{definition}

 \paragraph{Remark:} Since the linking number is a topological invariant, dilations and translations do not affect the linking number. This means that we may replace \eqref{eqn_property_L_1} with
 \begin{equation}\label{eqn_property_L_2}
 	v(\bf z) = \frac{u(T_r(\bf z)) - u(\bf x_0)}{r}
 \end{equation}
 and get an equivalent definition of property \hyperref[defn_property_L]{(L)}. 
 

 According to \cite[Proposition 4.1]{HenclMaly2010}, a Sobolev homeomorphism with positive Jacobian a.e.\ has property \hyperref[defn_property_L]{(L)} since we can take $T_r$ to be the (orientation-preserving) affine map $\bf z \mapsto \bf x_0 + r\bf z$ and show that $\L(v \circ \mu_{\bf a}, v \circ \nu_{\bf b}) = 1$ for any pair $(\bf a,\bf b)$ (see also \cite[Section 2.5]{GoldsteinHajlasz2019}).

\section{Calculations of the linking number}
 
 \paragraph{}
 The argument we use in the following lemma can be found in the proof of \cite[Theorem 1]{HenclOnninen2018} (where $\bf M$ is the identity matrix). For the convenience of the readers, we shall give a more detailed proof here. 
 
 \begin{lemma}\label{lemma_degree_near_isometry}
 	Let $\bf a \in B^{2}_{1/10}(\bf 0)$, $\bf b \in B^{2}(\bf 0) \cap B^{2}_{1/10}(\bf e_1)$ and suppose that $v\colon B_4(\bf 0) \to \R^n$ is continuous on $\mu_{\bf a}(S^1)$ and $\nu_{\bf b}(S^1)$. Let  $\bf{M} \in \Bbb M^{3\times 3}$ be an invertible linear transformation with $m_* := \inf_{|\bf z|=1} |\bf M[\bf z]|$, and set
 	\[
 	f(\bf z) := \left| v(\bf z)  - \bf M[\bf z] \right|.
 	\]
 	If
 	\begin{align*}
 		f(\mu_{\bf a}(\xi)) &= |v(\mu_{\bf a}(\xi)) -  \bf{M}[\mu_{\bf a}(\xi)]| \le m_*/10 \quad\text{ for all }\ \xi \in S^1 \quad{and} \\
 		f(\nu_{\bf b}(\et)) &= |v(\nu_{\bf b}(\et)) -  \bf{M}[\nu_{\bf b}(\et)]| \ \le m_*/10 \quad\text{ for all }\ \et \in S^1,
 	\end{align*}
 	then the linking number of the pair $(v \circ \mu_{\bf a}, v \circ \nu_{\bf b})$ is well-defined, and
 	\[
 	\L(v \circ \mu_{\bf a}, v \circ \nu_{\bf b}) = \sgn(\det \bf M).
 	\]
 \end{lemma}
 
 \begin{proof}
 	For each $\bf z \in \de \Bbb A \subset \R^3$, we may consider the homotopy $H\colon \de\Bbb A \times [0,1] \to \R^n$ defined by 
 	\[
 	H(\bf z,\tau) := (v_\tau \circ \mu_{\bf a})(\xi(\bf z)) - (v_\tau \circ \nu_{\bf b})(-\et(\bf z)),
 	\]
 	or equivalently
 	\[
 	H(\Ph(\xi,\et),\tau) := (v_\tau \circ \mu_{\bf a})(\xi) - (v_\tau \circ \nu_{\bf b})(-\et),
 	\]
 	where
 	\begin{align*}
 		v_\tau(\bf z) &:= (1-\tau)v(\bf z)  + \tau \bf M[\bf z] \\
 		&\ =  \bf M[\bf z] + (1-\tau)\left( v(\bf z)  -\bf M[\bf z] \right).
 	\end{align*}
 	It is easy to see that
 	\begin{align*}
 		H(\bf z,0) &=  (v \circ \mu_{\bf a})(\xi(\bf z)) - (v \circ \nu_{\bf b})(-\et(\bf z)), \quad\text{and} \\
 		H(\bf z,1) &= \bf M[\mu_{\bf a}(\xi)] - \bf M[\nu_{\bf b}(-\et)].
 	\end{align*}
 	
 	From \cite{HenclMaly2010}, we know that 
 	\[
 	\deg(H(\cdot, 1), \de \Bbb A, \bf 0) = \sgn(\det \bf M),
 	\]
 	hence the conclusion follows if we can show that $H( \de \Bbb A,\tau) \ne 0$ for all $\tau \in [0,1]$.
 	
 	Observe that for each $\xi\in S^1$ we have
 	\[
 	|\mu_{\bf a}(\xi) - \mu(\xi)| = |\Ph(\xi,{\bf a})- \Ph(\xi,\bf 0)| = | (a_1\xi_1, a_1\xi_2, a_2) | = |{\bf a}| < 1/10,
 	\]
 	and for each $\et \in S^1$ we have
 	\[
 	|\nu_{\bf b}(\et) - \nu(\et)| = |\Ph({\bf b},\et) - \Ph(\bf e_1,\et)| = | ((b_1-1)\et_1, b_2\et_1, (|b|-1)\et_2 ) + 2(\hat b_1 -1, \hat b_2,0) |, 
 	\]
 	thus
 	\[
 	|\nu_{\bf b}(\et) - \nu(\et)| \le |{\bf b} - \bf e_1| + 2|\bf{\hat b} - \bf e_1| < 5/10.
 	\]
 	This implies that $\dist(\mu_{\bf a}(S^1),\mu(S^1)) \le 1/10$ and $\dist(\nu_{\bf b}(S^1),\nu(S^1)) \le 5/10$ for any ${\bf a} \in E_{hor}$ and ${\bf b} \in E_{ver}$, hence 
 	\begin{align*}
 		\dist(\bf M[\mu_{\bf a}(S^1)], \bf M[\nu_{\bf b}(S^1)]) &\ge m_* \dist(\mu_{\bf a}(S^1), \nu_{\bf b}(S^1)) \\
 		&\ge m_*\left( \dist(\mu(S^1),\nu(S^1)) - 6/10  \right) \\
 		&= m_*(1- 6/10).
 	\end{align*}
 	Since 
 	\begin{align*}
 		H(\Ph(\xi,\et),\tau) &= (v_\tau \circ \mu_{\bf a})(\xi) - (v_\tau \circ \nu_{\bf b})(-\et) \\
 		&= \bf M[\mu_{\bf a}(\xi)] - \bf M[\nu_{\bf b}(-\et)] + (1- \tau) \left( v(\mu_{\bf a}(\xi)) - \bf{M}[\mu_{\bf a}(\xi)] \right) \\
 		&\quad  - (1-\tau) \left( v(\nu_{\bf b}(-\et)) - \bf{M}[\nu_{\bf b}(-\et)] \right),
 	\end{align*}
 	we have
 	\begin{align*}
 		|H(\Ph(\xi,\et),\tau)| &\ge  |\bf M[\mu_{\bf a}(\xi)] - \bf M[\nu_{\bf b}(-\et)]| - (1- \tau)f(\mu_{\bf a}(\xi)) - (1- \tau)f(\nu_{\bf b}(\et)) \\
 		&\ge \dist(\bf M[\mu_{\bf a}(S^1)], \bf M[\nu_{\bf b}(S^1)])  - 2m_*/10 \\
 		&\ge m_*(1-8/10) > 0.
 	\end{align*}
 	This shows that $H(\de \Bbb A,\tau) \ne \bf 0$ for all $\tau\in [0,1]$ as claimed. 	
 \end{proof}

 \paragraph{}
 Let $X$ and $Y$ be smooth, oriented manifolds without boundary embedded in some Euclidean spaces, and let $Z$ be a submanifold of $Y$. We shall denote the tangent space of $X$ at $x$ by $T_x(X)$, and for a submanifold $S$ of $X$, the normal space of $S$ with respect to $X$ at the point $x\in S$ is denoted by $N_x(S;X)$.

 \begin{definition}
 	Suppose $f\colon X\to Y$ is a smooth map. We say that $f$ is \emph{transversal} to $Z$ if
 	\[
 	Df_x(T_x(X)) + T_{f(x)}(Z) = T_{f(x)}(Y)
 	\]
 	for all $x \in f^{-1}(Z)$. If $X$ is a submanifold of $Y$, we say that $X$ intersects $Z$ transversally if the inclusion map $i\colon X \hookrightarrow Y$ is transversal to $Z$ (the roles of $X$ and $Z$ can be interchanged in that case).
 \end{definition}

 When $f$ is transversal to $Z$, $S := f^{-1}(Z)$ is a submanifold of $X$, where the codimension of $S$ in $X$ is equal to the codimension of $Z$ in $Y$. It can be seen that $Df_x$ maps $N_x(S;X)$ isomorphically to $Df_x(N_x(S;X))$ such that
 \[
 Df_x(N_x(S;X)) \oplus T_{f(x)}(Z) = T_{f(x)}(Y),
 \]
 thus the orientation on $Y$ induces an orientation on $Df_x(N_x(S;X))$, and hence on $N_x(S;X)$ via the isomorphism $Df_x$. The orientations on $X$ and $N_x(S;X)$ then induce the \emph{preimage orientation} on $S$ according to the relation
 \[
 N_x(S;X) \oplus T_x(S) = T_x(X).
 \]

 \begin{definition}
 	When $X$ is compact, $Z$ is closed, $\dim X + \dim Z = \dim Y$ and $f$ is transversal to $Z$, $S = f^{-1}(Z)$ is a set of finite isolated points. The preimage orientation on $S$ assigns a number $+1$ or $-1$ to each point in $S$. The sum of these numbers are called the \emph{intersection number}  $I(f,Z)$.
 \end{definition}
 See \cite{GuilleminPollack1974} or \cite{Hirsch1976} for a more complete discussion of the intersection number and the derivations of the facts given above.

 \paragraph{}
 We shall now show how to calculate the linking number when one curve in the link is a planar curve using the $2$-dimensional topological degree. To review some standard facts on this topic, see \cite{Deimling1985} or \cite{Zeidler1986}. Recall that for an open set $U$, a point $y \in \R^n$ and a continuous map $f\in C(\lbar U;\R^n)$, the degree $\deg(f,U,y)$ depends only on the value of $f$ on $\de U$, so we will often write $\deg(f,\de U, y)$ in place of $\deg(f, U, y)$.

 \begin{definition}\label{defn_disks_and_circles_in_H_z}
 	For each $\bf a = (a_1,a_2) \in B^{2}_{1/10}(\bf 0)$, the circle $\mu_{\bf a}(S^1)$ lies entirely in the plane
 	\[
 	H_{a_2} := \{(x,y,a_2) \in\R^3 : x,y \in\R \}.
 	\]
 	We shall let $\td D_{\bf a}$ be the (open) disk in $H_{a_2}$ such that $\de \td D_{\bf a} = \mu_{\bf a}(S^1)$. For each $\bf b \in B^{2}(\bf 0) \cap B^{2}_{1/10}(\bf e_1)$, the circle $\nu_{\bf b}(S^1)$ intersects the plane $H_{a_2}$ twice, we let $\td{\bf p}_{\bf a \bf b}$ be the point that lies inside $\td D_{\bf a}$ and let $\td{\bf q}_{\bf a \bf b}$ be the point that lies outside the disk. 
 	
 	Lastly, let $\pi_{xy}\colon\R^3 \to \R^2$ be the projection onto the first two coordinates. We shall write $S_{\bf a} := \pi_{xy}(\mu_{\bf a}(S^1))$, $D_{\bf a} := \pi_{xy}(\td D_{\bf a})$, $\bf p_{\bf a \bf b} := \pi_{xy}(\td{\bf p}_{\bf a \bf b})$ and $\bf q_{\bf a \bf b} := \pi_{xy}(\td{\bf q}_{\bf a \bf b})$.
 \end{definition}

 \begin{lemma}\label{lemma_linking_number_using_2D_degree}
 	Let $v\colon B_4(\bf 0) \to \R^3$ and suppose that there exist $z' \in \R$ and a map $w\colon B_4^2(\bf 0) \to \R^2$ such that $v(z_1, z_2, z') = (w_1(z_1,z_2), w_2(z_1,z_2), 0)$ for all $z_1,z_2$. Let $\bf a\in B^{2}_{1/10}(\bf 0)$ and $\bf b \in B^{2}(\bf 0) \cap B^{2}_{1/10}(\bf e_1)$ be fixed, where $\bf a=(a_1,a_2)$ satisfies $a_2=z'$, and let $S_{\bf a}$, $\bf p_{\bf a\bf b}$ and $\bf q_{\bf a\bf b}$ be defined as in Definition \ref{defn_disks_and_circles_in_H_z}. Assume that $v = (v_1, v_2, v_3)$ is continuous on $\mu_{\bf a}(S^1)$\footnote{It then easily follows that $w$ is continuous on $S_{\bf a}$, hence the right hand side of \eqref{eqn_linking_number_using_2D_degree} is well-defined.} and $\nu_{\bf b}(S^1)$, and that $v_3$ is absolutely continuous with $\de_3 v_3 > 0$ a.e., then
 	\begin{equation}\label{eqn_linking_number_using_2D_degree}
 		\L(v\circ \mu_{\bf a}, v\circ \nu_{\bf b} ) = \deg(w, S_{\bf a}, w(\bf p_{\bf a \bf b})) - \deg(w, S_{\bf a}, w(\bf q_{\bf a \bf b})),
 	\end{equation}
 	provided that $v\circ \mu_{\bf a}(S^1) \cap v\circ \nu_{\bf b}(S^1) = \emptyset$.
 \end{lemma}
 
 \begin{proof}
 	In this proof, we shall identify $S^1$ with $[0,2\pi]$ and identify a function $f \colon S^1 \to \R^3$ with a periodic function $\td f \colon \R \to \R^3$, where $\td f(\th) = f(\cos(\th),\sin(\th))$ for $\th \in [0,2\pi]$ and extends $\td f$ by $\td f(\th + 2n\pi) =  \td f(\th)$ for $n \in \Bbb Z$. Hence, for a fixed bump function $\phi \in C^\oo(\R)$, we may define $f_\ve$ to be the function on $S^1$ that is identified with $\td f * \phi_\ve$, where $\phi_\ve := \frac{1}{\ve}\phi(\frac{\cdot}{\ve})$. Clearly, each $f_\ve$ is smooth and  $f_\ve \to f$ uniformly as $\ve \searrow 0$.
 	
 	Let $f := v\circ \mu_{\bf a}$, $g:= v\circ \nu_{\bf b}$ and let $(f_\ve)_\ve$ and $(g_\ve)_\ve$ be the corresponding families of mollified functions converging uniformly to $f$ and $g$, respectively. For each $\ve >0$, $N_\ve = g_\ve(S^1)$ is then a smooth $1$-dimensional manifold. Suppose that $v\circ \mu_{\bf a}(S^1) \cap v\circ \nu_{\bf b}(S^1) = \emptyset$, since the linking number is continuous with respect to uniform convergence, we have
 	\begin{align*}
 		\L(v\circ \mu_{\bf a}, v\circ \nu_{\bf b} ) &= \lim_{\ve \to 0} \L(f_\ve, g_\ve ) \\
 		&= \lim_{\ve \to 0} I(F_\ve, N_\ve), 
 	\end{align*}
 	where $F_\ve\colon \lbar{B_1^2(\bf 0)} \to \R^3 $ is an arbitrary extension of $f_\ve$ that is smooth on $B_1^2(\bf 0)$ with its image lying entirely in $H_{z'} = \{(x,y,z') \in\R^3 : x,y \in\R \}$, and $F_\ve$ is transversal to $N_\ve$ (see \cite[Chapter 2.3]{GuilleminPollack1974} for the existence of such extension). The fact that the linking numbers can be computed with the intersection numbers using the formula $\L(f_\ve, g_\ve ) = I(F_\ve, N_\ve)$ can be seen in \cite[Chapter III.17]{BottTu1982} (see also \cite[Chapter 5.D]{Rolfsen1976}).

 	We let $P_\ve = N_\ve \cap H_{z'}$. $P_\ve$ consists of exactly two points, i.e.\ $P_\ve = \{\bf p_\ve, \bf q_\ve \}$, where $\bf p_\ve$ and $\bf q_\ve$ satisfy $\bf p_\ve \to v(\td{\bf p}_{\bf a \bf b})$ and $\bf q_\ve \to v(\td{\bf q}_{\bf a \bf b})$ as $\ve \to 0$, so $P$ is an oriented $0$-dimensional manifold whose orientation is induced from the orientation of $N_\ve$.\footnote{Here we endow $B_1^2(\bf 0)$, $\R^3$ and $H_{z'}$ with the standard orientations, and the orientation on $N_\ve$ is chosen so that the sign of $I(F_\ve, N_\ve)$ matches the sign of $\L(f_\ve, g_\ve )$. In particular, we have $x_3 >0$ for $\bf x \in T_{\bf p_\ve}(N_\ve)$ and $x_3 <0$ for $\bf x \in T_{\bf q_\ve}(N_\ve)$. This orientation on $N_\ve$ is consistent as $\ve$ varies.}  It can be verified that $\bf p_\ve$ is the point with positive orientation number and $\bf q_\ve$ is the one with negative orientation number.	Since $F_\ve$ is transversal to $N_\ve$, it is also transversal to $P_\ve$, and hence $I(F_\ve, P_\ve) = I(F_\ve, N_\ve)$, where on the left hand side we view $F_\ve$ as a map into a $2$-dimensional manifold $H_{z'}$. 
 	
 	Since the intersection number $I(F_\ve, P_\ve)$ coincides with the topological degree (or the \emph{winding number}, see \cite[Chapter 2.5, 3.6]{GuilleminPollack1974} and \cite[Chapter 5]{Hirsch1976}), we have
 	\begin{align*}
 		 	I(F_\ve, N_\ve) &= I(F_\ve, P_\ve) \\
 		 	&= \deg(\pi_{xy} \circ F_\ve, B_1^2(\bf 0), \pi_{xy}(\bf p_\ve)) - \deg(\pi_{xy} \circ F_\ve, B_1^2(\bf 0), \pi_{xy}(\bf q_\ve)) \\
 		 	&= \deg(\pi_{xy} \circ f_\ve, S^1, \pi_{xy}(\bf p_\ve)) - \deg(\pi_{xy} \circ f_\ve, S^1, \pi_{xy}(\bf q_\ve))
 	\end{align*}
 	Since $\bf p_\ve \to v(\td{\bf p}_{\bf a \bf b})$, $\bf q_\ve \to v(\td{\bf q}_{\bf a \bf b})$ and  $f_\ve \to v \circ \mu_{\bf a}$ uniformly as $\ve \to 0$, using also $\pi_{xy}\circ v = w \circ \pi_{xy}$, we may deduce
 	\begin{align*}
 		\lim_{\ve \to 0} I(F_\ve, P_\ve) &= \deg(\pi_{xy} \circ v \circ \mu_{\bf a}, S^1, \pi_{xy}(v(\td{\bf p}_{\bf a \bf b}))) - \deg(\pi_{xy} \circ v \circ \mu_{\bf a}, S^1, \pi_{xy}(v(\td{\bf q}_{\bf a \bf b}))) \\
 		&= \deg(w \circ \pi_{xy} \circ \mu_{\bf a}, S^1, w(\bf p_{\bf a \bf b}) ) - \deg(w \circ \pi_{xy} \circ \mu_{\bf a}, S^1, w(\bf q_{\bf a \bf b}) ) \\
 		&= \deg(w, S_{\bf a}, w(\bf p_{\bf a \bf b})) - \deg(w, S_{\bf a}, w(\bf q_{\bf a \bf b})).
 	\end{align*}
 	This completes the proof. 	
  \end{proof}

  \paragraph{Remark:} The assumption $\de_3 v_3 > 0$ ensures that that $N_\ve = g_\ve(S^1)$ intersects that plane $H_{z'}$ transversally in $\R^3$ and that the sign when calculating the linking number via the intersection number is correct.

 The following proposition is a direct consequence of Theorem 4.1 in \cite{BarchiesiHenaoMora-Corral2017}, adapted to fit our purpose.
 
 \begin{proposition}\label{prop_degree_for_map_satisfying_div_identities}
 	Let $\Lambda\subset \R^2$ be an open domain and $w\in \W{1,p}(\Lambda;\R^2)$ for some $p>1$. Suppose that $\det Dw > 0$ a.e.\ and $w$ satisfies that divergence identities \eqref{div_identity}, then, for almost every $r>0$ such that $B_r^2(\bf 0) \subset \Lambda$, $\deg(w, \de B_r^2(\bf 0), \cdot)$ is defined (i.e.\ $w$ is continuous on $\de B_r^2(\bf 0)$) and we have
 	\begin{equation}\label{eqn_degree_for_map_satisfying_div_identities}
 		 \deg(w, \de B_r^2(\bf 0), \bf y) = \# \{ \bf x \in B_r^2(\bf 0) \cap \Lambda_0 : u(\bf x) = \bf y \}
 	\end{equation}
 	for almost every $\bf y\in \R^2$. 
 \end{proposition}
 
 Note that here $\# A$ denotes the cardinality of $A$, and $\Lambda_0$ is defined as in Definition \ref{defn_approx_diff_set}.
 
 \begin{proof}
 	It is known that for almost every $r>0$ such that $B_r^2(\bf 0) \subset \Lambda$, $B_r^2(\bf 0)$ is a `good open set' in the sense of \cite[Definition 2.17]{BarchiesiHenaoMora-Corral2017}\footnote{In particular, if an open set $U \subset \Lambda$ is good, then it has a $C^2$ boundary and $\res{w}_{\de U} \in \W{1,2}(\de U;\R^2)$, hence $w$ is continuous on $\de U$. The proof that $B_r^2(\bf 0)$ is a good open set for almost every $r>0$ can be found in \cite[Lemma 2]{HenaoMora-Corral2011} (see also \cite{MullerSpector1995}).}, hence \cite[Theorem 4.1]{BarchiesiHenaoMora-Corral2017} implies that
 	\[
 	\deg(w, \de B_r^2(\bf 0), \bf y) = \# \{ \bf x \in B_r^2(\bf 0) \cap \Lambda_d : u(\bf x) =\bf  y \} \quad\text{for a.e.\ }\bf y
 	\]
 	for all such $r$. The right hand side of this equation differs from \eqref{eqn_degree_for_map_satisfying_div_identities} precisely when $\bf y \in w((\Lambda_d\bsl\Lambda_0) \cap B_r^2(\bf 0))$, but $(\Lambda_d\bsl\Lambda_0) \cap B_r^2(\bf 0)$ is a null set, so $w((\Lambda_d\bsl\Lambda_0)\cap B_r^2(\bf 0))$ is also null. This concludes the proof. 	
 \end{proof}

 \section{Main results}

 \paragraph{}
 The proof of our first main theorems follows closely the line of reasoning given in \cite[Theorem 1]{HenclOnninen2018} by Hencl and Onninen. Our main contribution is in recognising that while the linking number is a topological invariant, one does not need to assume that each Sobolev map in the sequence $(u_j)_j$ is an orientation-preserving homeomorphism to reach the same conclusion. A careful analysis of the proof of \cite[Theorem 1]{HenclOnninen2018} leads us to define property \hyperref[defn_property_L]{(L)}, which turns out to be a sufficient condition for the weak limit of $(u_j)_j$ to have non-negative Jacobian.

 \begin{theorem}\label{thm_main_for_property_L}
 	Let $\om \subset \R^3$ be a bounded open domain. $(u_j)_j$ be a sequence of $\W{1,2}(\om)$ maps that satisfy property \hyperref[defn_property_L]{(L)}. Suppose that $u_j \wlim u$ for some $u\in\W{1,2}(\om;\R^3)$, then $\det Du \ge 0$ a.e.
 \end{theorem}

 \begin{proof}
 	We define $S := \{ \det Du < 0 \}$, our goal in this proof is to show that $S$ is a null set using Lemma \ref{lemma_blow_up} from the appendix. Let $\bf x_0 \in S$ and let $\{ T_r \}_{r\in(0,R']}$ be the family of diffeomorphisms generated by $L=L_{\bf x_0}$ (defined in Definition \ref{defn_parametrization_L}) according to Lemma \ref{lemma_local_C^1_parametrization}.
 	
 	Following the notations in Lemma \ref{lemma_degree_near_isometry}, we shall write $\bf M = Du(\bf x_0)\circ \bf A$ (so $\det \bf M < 0$) and $m_* := \inf_{|\bf z|=1} | Du(\bf x_0)[\bf A z]|$. Let $\dt>0$ be a sufficiently small number.\footnote{At least small enough that Lemma \ref{lemma_size_of_good_set} can be applied, and less than, say, $1/40$.} From Lemma \ref{lemma_u_j_near_u_in_small_ball}, for each $r\in (0,R')$, there is a constant $j_0 = j_0(\bf x_0,\dt,r) \in \Bbb N$ such that
 	\begin{equation}
 		\int_{B_{4}(\bf 0)} \left| \frac{ u_j(T_r(\bf z)) - u(\bf x_0) }{r} - Du(\bf x_0)[\bf A\bf z] \right| \,dz < \dt^3
 	\end{equation}
 	for all $j \ge j_0$. We now let the value of $r < R'$ be fixed for the rest of this proof. To apply Lemma \ref{lemma_degree_near_isometry}, we let 
 	\[
 	v_j(\bf z):= \frac{u_j(T_r(\bf z)) - u(\bf x_0)}{r}
 	\] 
 	and set
 	\[
 	f_j(\bf z) := \left| v_j(\bf z) - Du(\bf x_0)[\bf A\bf z] \right|.
 	\] 
 	
 	Let $E_{hor}$ and $E_{ver}$ be the sets given by Definition \ref{defn_property_L}. We can see that for almost every ${\bf a} \in E_{hor}$, there exists $\xi \in S^1$ such that
 	\begin{equation}\label{thm_main_eqn1}
 		f_j(\mu_{\bf a}(\xi)) =\left| \frac{ u_j(T_r(\mu_{\bf a}(\xi))) - u(\bf x_0) }{r} - Du(\bf x_0)[\bf A\mu_{\bf a}(\xi)] \right| > m_*/10,
 	\end{equation}
 	or, for almost every ${\bf b} \in E_{ver} \subset B^{2} \cap B_p^{2}(e_1)$, there exists $\eta \in S^1$ such that
 	\begin{equation}\label{thm_main_eqn2}
 		f_j(\nu_{\bf b}(\et)) =\left| \frac{ u_j(T_r(\nu_{\bf b}(\et))) - u(\bf x_0) }{r} - Du(\bf x_0)[\bf A\nu_{\bf b}(\et)] \right| > m_*/10. 
 	\end{equation}
 	Indeed, let $\td A \subset E_{hor}$ be the set of all $\bf a \in E_{hor}$ such that \eqref{thm_main_eqn1} fails for all $\xi$,  and let $\td B \subset E_{ver}$ be the set of all $\bf b \in E_{ver}$ such that \eqref{thm_main_eqn2} fails for all $\et$. Then, $\td A$ and $\td B$ are measurable\footnote{We can write $\td A = \{ \bf a \in E_{hor} : g(\bf a) \le m_*/10 \}$, where $g(\bf a) := \sup_{i\in \Bbb N} f_j(\mu_{\bf a}(\xi_i)$ for some sequence $(\xi_i)_i$ that is dense in $S^1$. Since $g$ is measurable, $\td A$ is thus measurable as well. A similar argument shows that $\td B$ is measurable.}	and for each pair $(\bf a,\bf b) \in \td A \times \td B$, Lemma \ref{lemma_degree_near_isometry} implies that
 	\[
 	\L\left( v_j \circ \mu_{\bf a}, v_j \circ \nu_{\bf b} \right) = \sgn(\det Du(\bf x_0) \circ \bf A) = -1,
 	\]
 	so we must have $|\td A| = 0$ or $|\td B| = 0$ since $u_j$ satisfies property \hyperref[defn_property_L]{(L)}.\footnote{See the remark that follows Definition \ref{defn_property_L}. } Without loss of generality, we assume that it is the former case and then redefine $E_{hor}$ so that $\td A = \emptyset$.

 	For each $j \ge j_0$, we define 
 	\begin{equation}	\label{thm_main_eqn3}
 		I_{A,j} :=\left\{ {\bf a} \in E_{hor} : \Hd{1}\left(\{ \res{f_j}_{\mu_{\bf a}(S^1)} \ge m_*\dt \} \right) < 2\dt \right\}
 	\end{equation}
 	Note that $\mu_{\bf a}(S^1) = \et^{-1}({\bf a})$, hence we may apply Lemma \ref{lemma_size_of_good_set} to conclude that $\L^{2}(I_{A,j}) > \frac{1}{2} \L^{2}(E_{hor}) = \frac{1}{2} \L^{2}(B_{1/10}^2(\bf 0))$, provided that our given $\dt$ is sufficiently small.\footnote{Here, we fix $\al = 1/2$ in Lemma \ref{lemma_size_of_good_set}. In the case that $|\td B| = 0$, the same analysis can be done to
 		\[
 		I_{B,j} :=\left\{ {\bf b} \in E_{ver} : \Hd{1}\left(\{ \res{f_j}_{\nu_{\bf b}(S^1)} \ge m_*\dt \} \right) < 2\dt \right\},
 		\]
 		using the fact that  $\nu_{\bf b}(S^1) = \td\xi^{-1}({\bf b})$ to conclude that $\L^{2}(I_{B,j}) > \frac{1}{2} \L^{2}(E_{ver})$ instead.} From the discussion in the previous paragraph, we have seen that for each ${\bf a} \in I_{A,j}$, there exists $\xi =\xi_{\bf a}$ such that \eqref{thm_main_eqn1} holds, i.e.
 	\[
 	f_j(\xi_{\bf a}) = \left| v_j(\mu_{\bf a}(\xi))  - Du(\bf x_0)[\bf A\mu_{\bf a}(\xi_{\bf a})] \right| > m_*/10.
 	\]
 	

 	From Morrey's embedding theorem, there exists a subset $I_{A,j}'$ of full measure in $I_{A,j}$ such that for each $\bf a \in I_{A,j}'$, we have 
 	\[
 	|v_j(\bf z) - v_j(\bf z') | \le C|\bf z - \bf z'|^{\frac{1}{2}} \left( \int_{\mu_{\bf a}(S^1)} |Dv_j|^{2} \,d\Hd{1} \right)^{\frac{1}{2}}
 	\]
 	for all $\bf z,\bf z' \in \mu_{\bf a}(S^1)$. We now pick $\bf z := \mu_{\bf a}(\xi_{\bf a})$ so that
 	\[
 	| v_j(\bf z) - Du(\bf x_0)[\bf A \bf z] | > m_*\p ,
 	\]
 	and we may then pick another $\bf z' \in \mu_{\bf a}(S^1)$ such that
 	\[
 	| v_j(\bf z') - Du(\bf x_0)[\bf A \bf z']  | < m_*\dt \quad \text{and} \quad |\bf z - \bf z'| < \dt.\footnote{It can be verified that $\Hd{1}(B_{\dt}(\bf z) \cap \mu_a(S^1)) \ge 2\dt$ uniformly for all  $\bf a \in E_{hor}$ whenever $\dt$ is sufficiently small, so not all points of $B_{\dt}(\bf z) \cap \mu_a(S^1)$ can be in the set $\{ \res{f_j}_{\mu_{\bf a}(S^1)} \ge m_*\dt \}$ since its $\Hd{1}$ measure is less than $2\dt$  (see \eqref{thm_main_eqn3}). Hence, some points in $B_{\dt}(\bf z) \cap \mu_a(S^1)$ must belong to $\{ \bf z' \in \mu_{\bf a}(S^1): f_j(z') < m_*\dt \}$.} 
 	\]
 	Thus
 	\begin{align*}
 		m_*/10 - m_*\dt - m_*\dt &\le f_j(\bf z) - f_j(\bf z') - m_*|\bf z - \bf z'| \\
 		&\le  | v_j(\bf z) -  Du(\bf x_0)[\bf A \bf z] | - | v_j(\bf z') -  Du(\bf x_0)[\bf A \bf z'] |   \\
 		&\quad\  - |Du(\bf x_0)[\bf A \bf z] - Du(\bf x_0)[\bf A \bf z']| \numberthis\label{thm_main_eqn4} \\
 		&\le |v_j(\bf z) - v_j(\bf z') | \numberthis\label{thm_main_eqn5}\\
 		&\le C|\bf z - \bf z'|^{\frac{1}{2}}  \left( \int_{\mu_{\bf a}(S^1)} |Dv_j|^{2} \,d\Hd{1} \right)^{\frac{1}{2}},
 	\end{align*}
 	where we used the definition of $m_*$ in \eqref{thm_main_eqn4}, and the standard triangle inequality $|a|-|b|-|c| \le |a\pm b \pm c|$ in \eqref{thm_main_eqn5}. 
 	
 	Therefore, since $|\bf z - \bf z'| < \dt$, for all sufficiently small $\dt>0$ we have
 	\[
 	1 \le C\dt  \int_{\mu_{\bf a}(S^1)} |Dv_j|^{2} \,d\Hd{1} .
 	\]
 	We may integrate the above with respect to ${\bf a} \in I_{A,j}' \subset B^{2}_{1/10}(\bf 0)$ using the coarea formula to get 
 	\[
 	\dt^{-1} \int_{B^{2}_{1/10}(\bf 0)} \chi_{I_{A,j}'}(\bf a) \,d\L^{2}(\bf a)  \le C \int_{B^{2}_{1/10}(\bf 0)} \int_{\mu_{\bf a}(S^1)} |Dv_j|^{2} \,d\Hd{1} \,d\L^{2}(\bf a),
 	\]
 	which implies
 	\[
 	\frac{1}{2} \L^2(B_{1/10}^2(\bf 0)) \dt^{-1} \le C \int_{B_{4}(0)} |Dv_j(\bf z)|^{2} \,d\bf z.
 	\]
 	Since $Dv_j(\bf z) = \frac{1}{r} Du_j(T_r(\bf z)) \circ DT_r(\bf z)$, we have 
 	\begin{align*}
 		\int_{B_{4r}(\bf x_0)} | Du_j(\bf x) |^{2} \,d\bf x &= \int_{U_r} |Du_j(T_r(\bf z)) |^{2} \det DT_r(\bf z) \,d\bf z \\
 		&= \int_{U_r} \left| Dv_j(\bf z) \circ \left[ \frac{DT_r(\bf z)}{r} \right]^{-1} \right|^{2} \det DT_r(\bf z) \,d\bf z \\
 		&\ge c^{-1} r^3  \int_{B_4(\bf 0)} \frac{\left| Dv_j(\bf z) \right|^{2}}{\left| DT_r(\bf z)/r \right|^{2}} \,d\bf z  \\
 		&\ge c^{-1} \left( \frac{2}{3|A|} \right)^{2} r^3  \int_{B_4(\bf 0)} \left| Dv_j(\bf z) \right|^{2} \,d\bf z,
 	\end{align*}
 	where the last inequality follows directly from Lemma \ref{lemma_local_C^1_parametrization} (iii). Thus
 	\begin{equation}\label{thm_main_eqn_6}
 		\dt^{-1} r^3 \le C \int_{B_{4r}(\bf x_0)} |Du_j(\bf x)|^{2} \,d\bf x.
 	\end{equation}

 	By passing to a subsequence, we may assume that $(\L^n \llcorner |Du_j|)\wslim \mu$ for some $\mu \in \M(\om)$. Since \eqref{thm_main_eqn_6} holds for all $r\in (0,R')$ (for sufficiently large $j \ge j_0(\bf x_0,\dt,r)$), we may pick $r$ such that $\mu(\de B_{4r}(\bf x_0)) = 0$, hence Lemma \ref{lemma_blow_up} can be applied. This concludes the proof.
 \end{proof}

 \paragraph{}
 Now, using the results from Section 5, we shall prove that an orientation-preserving generalised axisymmetric map that is one-to-one a.e.\ satisfies property \hyperref[defn_property_L]{(L)}. Let us recall that we define  $O_\th := \{ (r, \th, z)_{cyl} : r>0, z \in \R \}$ and $H_{z} := \{(x,y,z) \in\R^3 : x,y \in\R \}$.

 \begin{theorem}\label{thm_axisym_maps_satisfy_condition_L}
 	Let $\om \subset \R^3$ be a bounded open domain and let $u \in \W{1,2}(\om;\R^3)$ be a generalised axisymmetric map with $\det Du > 0 $ a.e. and is one-to-one a.e. Then $u$ satisfies property \hyperref[defn_property_L]{(L)}. 
 \end{theorem}
 
 \begin{proof}
 	Let $\bf x_0 = (r_0,\th_0,z_0)_{cyl} \in \om$ be a point not on the axis $\{ (0, 0, z) : z \in \R \}$ and let $\td u_1, \td  u_2$ and $\Theta$ be the maps defined in Definition \ref{defn_axisym}. Let $L=L_{\bf x_0}$ be the map defined as in Definition \ref{defn_parametrization_L} and $\{ T_r \}_{r\in(0,R']}$ be the family of diffeomorphisms generated by it according to Lemma \ref{lemma_local_C^1_parametrization}. We fix $r>0$ and let $v(\bf z)$ be defined as in Definition \ref{defn_property_L}, i.e.\
 	\[
 	v(\bf z) = u(T_r(\bf z)).
 	\]
 	Since $r$ is fixed, $T_r(\bf z) = L_{\bf x_0}^{-1}(c\bf z)$ for some constant $c>0$.
 	
 	Since $v \in \W{1,2}(B_4(\bf 0); \R^3)$, it follows that for $\L^2$-a.e. $\bf a \in B^{2}_{1/10}(\bf 0)$, the restriction of $v$ to $\mu_{\bf a}(S^1)$ is absolutely continuous (see e.g. \cite[Proposition 2.8]{MullerSpector1995} and the remarks that follow). Similarly, for $\L^2$-a.e.  $\bf b \in B^{2}_1(\bf 0) \cap B^{2}_{1/10}(\bf e_1)$, the restriction of $v$ to  $\nu_{\bf b}(S^1)$ is absolutely continuous, thus condition (i) of property \hyperref[defn_property_L]{(L)} is satisfied. We will show that for almost every $z \in (-4,4)$ and for $\L^1$-a.e.\ $\bf a$ such that $\mu_{\bf a}(S^1) \subset H_{z}$, we have
 	\[
 	\L(v \circ \mu_{\bf a}, v\circ \nu_{\bf b}) = 1
 	\]
 	for $\L^2$-a.e.\ $\bf b \in B^{2}_1(\bf 0) \cap B^{2}_{1/10}(\bf e_1)$ whenever the linking number is well-defined. This will, by Fubini's theorem, imply that condition (ii) of property \hyperref[defn_property_L]{(L)} holds for $v$.
 	
 	It is known that for almost every $z' \in (-4,4)$, $\res{v}_{H_{z'}}$ is in $\W{1,2}(\B_4(\bf 0) \cap H_{z'}; \R^3)$ (see also \cite[Proposition 2.8]{MullerSpector1995}). Since $v$ is one-to-one $\L^3$-a.e., $\res{v}_{H_{z'}}$ is one-to-one $\L^2$-a.e.\ for almost every $z' \in (-4,4)$ as well. For the rest of the proof, we shall let $z'$ be a fixed number such that  $\res{v}_{H_{z'}}$ has these properties. Since the linking number is invariant under a rotation, i.e.
 	\[
 	\L(v \circ \mu_{\bf a}, v\circ \nu_{\bf b}) = \L(\bf R\circ v \circ \mu_{\bf a},\bf R \circ v\circ \nu_{\bf b})
 	\]
 	when $\bf R \in SO(3)$), we will pick a suitable rotation $\bf R$ so that Lemma \ref{lemma_linking_number_using_2D_degree} can be applied.
 	
 	Let $\Lambda := \pi_{xy}(\B_4(\bf 0) \cap H_{z'}) \subset \R^2$. From the definition of generalised axisymmetric maps, it is clear that the image of $v(\cdot,\cdot,z')$ lies entirely in $O_{\Theta(\th_0 + cz')}$.  By performing a rotation by the matrix
 	\[
 	\bf R = \begin{bmatrix}
 		\cos \Theta(\th_0 + cz') &\sin \Theta(\th_0 + cz') &0 \\
 		0 &0 &-1 \\
 		-\sin\Theta(\th_0 + cz') &\cos\Theta(\th_0 + cz') &0
 	\end{bmatrix},
 	\]
 	the image of $(\bf R \circ v)(\cdot,\cdot,z')$ now lies in the plane $\{(x,y,0) : x,y\in \R \}$, i.e.\ $(\bf R \circ v)(\cdot,\cdot,z') = (w_1(\cdot,\cdot), w_2(\cdot,\cdot), 0)$ for a corresponding $w\colon\Lambda \to \R^2$.\footnote{It is not hard to see that $w_1(x,y) = \td u_1(r_0 + cx, \th_0 + cz',  z_0 - cy)$ and $w_2(x,y) = -\td u_2(r_0 + cx, \th_0 + cz',  z_0 - cy)$.}	A straightforward computation (see e.g.\ the appendix of \cite{HenaoRodiac2018}) shows that $\det Dv > 0$ a.e. and $\de_3 (\bf R \circ v) > 0$ a.e.\ (since $\Theta' >0$ a.e.), thus $\det Dw > 0$ a.e. Since we knew that $\res{v}_{H_{z'}}$ is in $\W{1,2}(\B_4(\bf 0) \cap H_{z'}; \R^3)$, it follows that $w \in \W{1,2}(\Lambda;\R^2)$ for all such $z'$, and hence $w$ satisfies the divergence identities \eqref{div_identity} according to \cite[Theorem 3.2]{MullerTangYan1994} (see also \cite[Lemma 2]{Mueller1988}). Note that since $\res{v}_{H_{z'}}$ is one-to-one $\L^2$-a.e., $w$ is one-to-one on $\Lambda_0$ (defined as in Definition \ref{defn_approx_diff_set}).\footnote{In fact, the result \cite{VodopyanovGoldshtein1977} by Vodop'yanov and Gol'dshtein shows that $w$ is even continuous on $\Lambda$.}
 	
 	For our fixed $z'$, we consider the set 
 	\[
 	E_{z'} := \{ \bf a = (a_1,a_2) \in B^{2}_{1/10}(\bf 0) : a_2 = z' \}
 	\]
 	and recall that $S_{\bf a} = \pi_{xy}(\mu_{\bf a}(S^1))$. A simple inspection of Definition \ref{defn_mu_a_and_nu_b} shows that for $\bf a = (a_1,a_2)$, $S_{\bf a} = \de B_{a_1+2}^2(\bf 0)$. By Proposition \ref{prop_degree_for_map_satisfying_div_identities}, we have 
 	\[
 	\deg(w, S_{\bf a}, \bf y) = \# \{ \bf x \in B_{a_1+2}^2(\bf 0) \cap \Lambda_0 : u(\bf x) = \bf y \} \quad\text{for a.e.\ }\bf y
 	\]
 	for $\L^1$-a.e.\ $\bf a \in E_{z'}$. Since $\Lambda_0$ is of full measure in $\Lambda$, it is not hard to verify that for $\L^2$-a.e.  $\bf b \in B^{2}_1(\bf 0) \cap B^{2}_{1/10}(\bf e_1)$, both $\bf p_{\bf a \bf b}$ and $\bf q_{\bf a \bf b}$ (from Definition \ref{defn_disks_and_circles_in_H_z}) belong to $\Lambda_0$. For any such $\bf b$, either $w(S_{\bf a}) \cap \{ w(\bf p_{\bf a \bf b}), w(\bf q_{\bf a \bf b}) \} \ne \emptyset$ and the linking number $\L(v \circ \mu_{\bf a}, v\circ \nu_{\bf b})$ is undefined, or $w(S_{\bf a}) \cap \{ w(\bf p_{\bf a \bf b}), w(\bf q_{\bf a \bf b}) \} = \emptyset$ and we have
 	\[ 
 	\deg(w, S_{\bf a}, w(\bf p_{\bf a \bf b})) = \# \{ \bf p_{\bf a \bf b} \} = 1
 	\]
 	since $w$ is one-to-one on $\Lambda_0$ and $\bf p_{\bf a \bf b} \in B_{a_1+2}^2(\bf 0)$, whereas
 	\[
 	\deg(w, S_{\bf a}, w(\bf q_{\bf a \bf b})) = \# \{ \} = 0
 	\]
 	since $\bf q_{\bf a \bf b}$ is the only point in $\Lambda_0$ that get mapped by $w$ to $w(\bf q_{\bf a \bf b})$, but $\bf q_{\bf a \bf b}$ is not in $ B_{a_1+2}^2(\bf 0)$. Thus, by Lemma \ref{lemma_linking_number_using_2D_degree},
 	\begin{align*}
 		 	\L(v \circ \mu_{\bf a}, v\circ \nu_{\bf b}) &= \L(\bf R\circ v \circ \mu_{\bf a},\bf R \circ v\circ \nu_{\bf b}) \\
 		 	&= \deg(w, S_{\bf a}, w(\bf p_{\bf a \bf b})) - \deg(w, S_{\bf a}, w(\bf q_{\bf a \bf b})) \\
 		 	&= 1. \numberthis \label{eqn_axisym_maps_satisfy_condition_L_1}
 	\end{align*}
 	
 	Lastly, we assume for contradiction that there exist measurable sets $\td A$ with $\L^2(\td A) >0$ and $\td B$ with $\L^2(\td B)>0$ such that 
 	\[
 	\L(v \circ \mu_{\bf a}, v \circ \nu_{\bf b}) < 0
 	\]
 	for every pair $(\bf a,\bf b) \in \td A \times \td B$. By Fubini's theorem, there exists a set $I \subset (-4,4)$ of positive measure such that for each $z \in I$, we can find another set $\td E \subset E_{z}$ of positive $\L^1$-measure such that 
 	\[
 	\L(v \circ \mu_{\bf a}, v \circ \nu_{\bf b}) < 0
 	\]
 	for all $(\bf a,\bf b) \in \td E \times \td B$. This, however, contradicts the fact that \eqref{eqn_axisym_maps_satisfy_condition_L_1} holds for $\L^1$-a.e. $\bf a \in E_z$ and $\L^2$-a.e. $\bf b$ whenever $\L(v \circ \mu_{\bf a}, v \circ \nu_{\bf b})$ is defined. Therefore, condition (ii) of property \hyperref[defn_property_L]{(L)} is satisfied and we are done.
 \end{proof}

 \paragraph{}
 Theorem \ref{thm_main} is then a direct consequence of Theorem \ref{thm_axisym_maps_satisfy_condition_L} and Theorem \ref{thm_main_for_property_L}.
 
 \begin{proof}[Proof of Theorem \ref{thm_main}]
 	Let $(u_j)_j$ be a sequence of generalised axisymmetric maps in $\W{1,2}(\om;\R^3)$ that are one-to-one a.e.\ and satisfy $\det Du_j >0$ a.e. and $u_j \wlim u$. By Theorem \ref{thm_axisym_maps_satisfy_condition_L}, each $u_j$ satisfy property \hyperref[defn_property_L]{(L)}, thus $\det Du \ge 0$ a.e.\ according to Theorem \ref{thm_main_for_property_L}.
 \end{proof}

\section{Appendix}

 \paragraph{The coarea formula:} Let $\xi\colon \R^3 \to \R^2$ be a Lipschitz function and denotes $J_\xi := \sqrt{\det ((D\xi)(D\xi)^t)}$. Let $g \colon \R^3 \to \R$ be a measurable function, then
 \[
 \int_{\R^3} g J_\xi \,d\L^3 = \int_{\R^2} \left( \int_{\xi^{-1}(\{\bf y\})} g \,d\Hd{1} \right) \,d\L^2(\bf y).
 \]  
 In particular, by letting $g = \chi_A$ for some measurable set $A \subset \R^3$, we have
 \[
 \int_{A} J_\xi \,d\L^3 = \int_{\R^2} \Hd{1}(A \cap \xi^{-1}(\{\bf y\})) \,d\L^2(\bf y).
 \]

 \paragraph{}
 Let $\om \subset \R^3$ be an open domain. For any $\mu \in \M(\om)$, we may define a maximal function
 \[
 M\!\mu(\bf x) := \sup_{r>0} \frac{\mu(B_r(\bf x) \cap \om)}{|B_r(\bf x)|}.
 \]
 It is known that $M\!\mu(\bf x) < \oo$ for a.e.\ $\bf x\in \om$. Indeed, from \cite[Theorem 1 (a), Section 3.1, chapter I]{Stein1993_HA_book}, let
 \[
 M\!f(\bf x) := \sup_{r>0} \frac{1}{|B_r(\bf x)|} \int_{\R^3} |f| \,d\mu,
 \]
 then $M\!f$ is finite a.e. We can simply take $f = \chi_\om$ in our case.

 For any sequence $(g_j)_j$ in $\Leb{1}(\om)$ such that $\sup_j \int_\om |g_j| \,dx < \oo$, its associated sequence of measures $(\L^3 \llcorner g_j)_j$ is also bounded in $\M(\om)$, hence there exists a subsequence (not relabelled) such that $(\L^3 \llcorner g_j)\wslim \mu$ for some $\mu \in \M(\om)$. Here we view $\M(\om) = C_0(\om)^*$. Recall also that, for $B_r(\bf x) \subset \om$, if $\mu(\de B_r(\bf x)) = 0$ and $\mu_j \wslim \mu$ in $\M(\om)$, then
 \[
 \mu_j(B_r(\bf x)) \to \mu(B_r(\bf x)).
 \]
 See e.g.\ \cite[Proposition 1.62 (b)]{AmbrosioFuscoPallara2000_BV_book}.
 
 \begin{lemma}\label{lemma_blow_up}
 	Let $\om\subset \R^3$ be an open domain and let $S \subset \om$ be its measurable subset. Let $(g_j)_j$ be a sequence in $\Leb{1}(\om)$ such that $(\L^n \llcorner g_j) \wslim \mu$ for some $\mu \in \M(\om)$. Suppose that for every $\bf x_0\in S$, there exists $C=C(\bf x_0)>0$ such that for any given (sufficiently small) $\dt>0$, we can find $r=r(\bf x_0,\dt)>0$ and $j_0 = j_0(\bf x_0,\dt,r) \in \Bbb N$ such that $B_{r}(\bf x_0)\subset \om$, $\mu(\de B_r(\bf x_0)) = 0$, and
 	\[
 	\dt^{-1} r^3 \le C \int_{B_r(\bf x_0)} |g_j| \,d\bf x
 	\]
 	for all $j \ge j_0$. Then $|S| = 0$.
 \end{lemma}
 
 \begin{proof}
 	Let $\bf x_0 \in S$. For a sufficiently small $\dt>0$, we pick $r>0$ and then $j_0 \in \Bbb N$ such that  $\mu(\de B_r(\bf x_0)) = 0$ and
 	\begin{align*}
 		\dt^{-1} \le C  \frac{1}{|B_r(\bf x)|} \int_{B_r(\bf x_0)} |g_j| \,d\bf x
 	\end{align*}
 	for all $j \ge j_0$. Taking the limit as $j \to \oo$, we have
 	\begin{align*}
 		\dt^{-1} \le C  \frac{ \mu(B_r\bf (x_0)) }{|B_r(\bf x)|} \le C M\!\mu(\bf x_0).
 	\end{align*}
 	We now let $\dt \searrow 0$ to conclude that $M\!\mu(\bf x_0) = \oo$. This shows that $S$ is a null set since $M\!\mu < \oo$ a.e.
 \end{proof}

 \begin{lemma}\label{lemma_size_of_good_set}
 	Let $f \colon B_4(\bf 0)\to \R$ be a positive measurable function and let $\et \colon B_4(\bf 0)\to \R^2$  be a Lipschitz function. Let $E \subset \R^2$ be a measurable set and $J_\et \le M$ on $\et^{-1}(E)$ for some $M > 0$. 
 	
 	For a fixed $m>0$ and any given $\dt>0$, define
 	\[
 	I :=\left\{\bf s \in E : \Hd{1}\left(\{ \res{f}_{\et^{-1}(\{\bf s\})} \ge m\dt \} \right) < 2\dt \right\}.
 	\] 
 	If $\dt >0$ is sufficiently small, then, for any given $\al\in(0,1)$,
 	\begin{equation}
 		\int_{B_4(\bf 0)} f \,d\L^3 < \dt^3, \label{lemma_size_of_good_set_eqn_0}
 	\end{equation}
 	implies that $\L^2(I) > (1-\al) \L^2(E)$.
 \end{lemma}

 \begin{proof}
 	Suppose for a contrary that $\L^2(E\bsl I) > \al \L^2(E)$, then
 	\begin{align*}
 		\int_{B_4(\bf 0)} f \,d\L^3 &=  \int_{\R^2} \left( \int_{\et^{-1}(\{\bf s\})} f \frac{1}{J_\et} \,d\Hd{1} \right) \,d\L^2(\bf s) \\
 		&\ge  \int_{E\bsl I} \left( \int_{ \{ \res{f}_{\et^{-1}(\{\bf s\})} \ge m\dt \} } f \frac{1}{J_\et} \,d\Hd{1}  \right) \,d\L^2(\bf s) \\
 		&\ge  \int_{E\bsl I} \frac{m\dt}{M} \Hd{1}\left(\{ \res{f}_{\et^{-1}(\{\bf s\})} \ge m\dt \} \right)   \,d\L^2(\bf s) \\
 		&\ge  \frac{2m\dt^2}{M} \L^2(E\bsl I) \\
 		&> \frac{2m\al\dt^2}{M}\L^2(E).
 	\end{align*}
 	However,  $\frac{2m\al\dt^2}{M}\L^2(E) > \dt^3$ whenever $\dt < \frac{2m\al}{M} \L^2(E)$, thus   \eqref{lemma_size_of_good_set_eqn_0} is violated. Therefore, we must have $\L^2(E\bsl I) \le \al \L^2(E)$ for all such $\dt>0$.
 \end{proof}

 \paragraph{Acknowledgements}
 D. Henao was supported by FONDECYT grant N. 1231401 and by Center for Mathematical Modeling, FB210005, Basal ANID Chile.

	\bibliography{Jacobian_of_generalized_axisymetric}

\begin{thebibliography}{10}

\bibitem{AmbrosioFuscoPallara2000_BV_book}
{\sc L.~Ambrosio, N.~Fusco, and D.~Pallara}, {\em Functions of bounded
  variation and free discontinuity problems}, Oxford Math. Monogr., Oxford:
  Clarendon Press, 2000.

\bibitem{Ball1977}
{\sc J.~M. Ball}, {\em Convexity conditions and existence theorems in nonlinear
  elasticity}, Arch. Ration. Mech. Anal., 63 (1977), pp.~337--403.

\bibitem{BallCurrieOlver1981}
{\sc J.~M. Ball, J.~C. Currie, and P.~J. Olver}, {\em Null {Lagrangians}, weak
  continuity, and variational problems of arbitrary order}, J. Funct. Anal., 41
  (1981), pp.~135--174.

\bibitem{BarchiesiHenaoMora-Corral2017}
{\sc M.~Barchiesi, D.~Henao, and C.~Mora-Corral}, {\em Local invertibility in
  {S}obolev spaces with applications to nematic elastomers and
  magnetoelasticity}, Arch. Ration. Mech. Anal., 224 (2017), pp.~743--816.

\bibitem{BarchiesiHenaoMora-CorralRodiac2023_harmonic}
{\sc M.~Barchiesi, D.~Henao, C.~Mora-Corral, and R.~Rodiac}, {\em Harmonic
  dipoles and the relaxation of the neo-{Hookean} energy in {3D} elasticity},
  Arch. Ration. Mech. Anal., 247 (2023), p.~46.

\bibitem{BarchiesiHenaoMora-CorralRodiac2024}
{\sc M.~Barchiesi, D.~Henao, C.~Mora-Corral, and R.~Rodiac}, {\em On the lack
  of compactness in the axisymmetric neo-{Hookean} model}, Forum Math. Sigma,
  12 (2024), p.~70.
\newblock Id/No e26.

\bibitem{BarchiesiHenaoMora-CorralRodiac2023_relaxation}
{\sc M.~Barchiesi, D.~Henao, C.~Mora-Corral, and R.~Rodiac}, {\em A relaxation
  approach to the minimisation of the neo-{H}ookean energy in 3{D}}.
\newblock SIAM J. Math. Anal., 2024.
\newblock (ArXiv:2311.02952).

\bibitem{BottTu1982}
{\sc R.~Bott and L.~W. Tu}, {\em Differential forms in algebraic topology},
  vol.~82 of Grad. Texts Math., Springer, Cham, 1982.

\bibitem{BouchalaHenclZhu2024}
{\sc O.~Bouchala, S.~Hencl, and Z.~Zhu}, {\em Weak limits of {S}obolev
  homeomorphisms are one to one}.
\newblock ArXiv:2409.01260, 2024.

\bibitem{Ciarlet1988}
{\sc P.~G. Ciarlet}, {\em Mathematical elasticity. {Volume} {I}:
  {Three}-dimensional elasticity}, vol.~20 of Stud. Math. Appl., Amsterdam
  etc.: North-Holland, 1988.

\bibitem{ContiDeLellis2003}
{\sc S.~Conti and C.~De~Lellis}, {\em Some remarks on the theory of elasticity
  for compressible {N}eohookean materials}, Ann. Sc. Norm. Super. Pisa Cl. Sci.
  (5), 2 (2003), pp.~521--549.

\bibitem{Deimling1985}
{\sc K.~Deimling}, {\em Nonlinear functional analysis}.
\newblock Berlin etc.: {Springer}-{Verlag}. {XIV}, 1985.

\bibitem{DolezalovaHenclMaly2023}
{\sc A.~Dole\v{z}alov\'{a}, S.~Hencl, and J.~Mal\'{y}}, {\em Weak {L}imit of
  {H}omeomorphisms in {$W^{1,n-1}$} and ({INV}) {C}ondition}, Arch. Ration.
  Mech. Anal., 247 (2023), p.~80.

\bibitem{DolezalovaHenclMolochanova2023}
{\sc A.~Doležalová, S.~Hencl, and A.~Molchanova}, {\em Weak limit of
  homeomorphisms in ${W}^{1,n-1}$: invertibility and lower semicontinuity of
  energy}.
\newblock ArXiv:2212.06452, 2023.

\bibitem{EvansGariepy1992_1st}
{\sc L.~C. Evans and R.~F. Gariepy}, {\em Measure theory and fine properties of
  functions}, Boca Raton: CRC Press, 1992.

\bibitem{Federer1996_GMT}
{\sc H.~Federer}, {\em Geometric measure theory.}, Class. Math., Berlin:
  Springer-Verlag, repr. of the 1969 ed.~ed., 1996.

\bibitem{FonsecaLeoniMaly2005}
{\sc I.~Fonseca, G.~Leoni, and J.~Mal{\'y}}, {\em Weak continuity and lower
  semicontinuity results for determinants}, Arch. Ration. Mech. Anal., 178
  (2005), pp.~411--448.

\bibitem{GiaquintaGiuseppeSoucek1989}
{\sc M.~Giaquinta, G.~Modica, and J.~Sou{\v{c}}ek}, {\em Cartesian currents,
  weak diffeomorphisms and existence theorems in nonlinear elasticity}, Arch.
  Ration. Mech. Anal., 106 (1989), pp.~97--159.

\bibitem{GiaquintaModicaSoucek1998_book}
\leavevmode\vrule height 2pt depth -1.6pt width 23pt, {\em Cartesian currents
  in the calculus of variations {I}. {Cartesian} currents}, vol.~37 of Ergeb.
  Math. Grenzgeb., 3. Folge, Berlin: Springer, 1998.

\bibitem{GoldsteinHajlasz2019}
{\sc P.~Goldstein and P.~Haj{\l}asz}, {\em Jacobians of {{\(W^{1,p}\)}}
  homeomorphisms, case {{\(p=[n/2]\)}}}, Calc. Var. Partial Differ. Equ., 58
  (2019), p.~28.
\newblock Id/No 122.

\bibitem{GuilleminPollack1974}
{\sc V.~Guillemin and A.~Pollack}, {\em Differential topology}.
\newblock Englewood {Cliffs}, {N}.{J}.: {Prentice}-{Hall}, {Inc}. {XVI}, 1974.

\bibitem{HenaoMora-Corral2010}
{\sc D.~Henao and C.~Mora-Corral}, {\em Invertibility and weak continuity of
  the determinant for the modelling of cavitation and fracture in nonlinear
  elasticity}, Arch. Ration. Mech. Anal., 197 (2010), pp.~619--655.

\bibitem{HenaoMora-Corral2011}
\leavevmode\vrule height 2pt depth -1.6pt width 23pt, {\em Fracture surfaces
  and the regularity of inverses for {BV} deformations}, Arch. Ration. Mech.
  Anal., 201 (2011), pp.~575--629.

\bibitem{HenaoRodiac2018}
{\sc D.~Henao and R.~Rodiac}, {\em On the existence of minimizers for the
  neo-{Hookean} energy in the axisymmetric setting}, Discrete Contin. Dyn.
  Syst., 38 (2018), pp.~4509--4536.

\bibitem{HenclMaly2010}
{\sc S.~Hencl and J.~Mal{\'y}}, {\em Jacobians of {Sobolev} homeomorphisms},
  Calc. Var. Partial Differ. Equ., 38 (2010), pp.~233--242.

\bibitem{HenclOnninen2018}
{\sc S.~Hencl and J.~Onninen}, {\em Jacobian of weak limits of {S}obolev
  homeomorphisms}, Adv. Calc. Var., 11 (2018), pp.~65--73.

\bibitem{Hirsch1976}
{\sc M.~W. Hirsch}, {\em Differential topology}, vol.~33 of Grad. Texts Math.,
  Springer, Cham, 1976.

\bibitem{IwaniecMartin2001}
{\sc T.~Iwaniec and G.~Martin}, {\em Geometric function theory and nonlinear
  analysis}, Oxford Math. Monogr., Oxford: Oxford University Press, 2001.

\bibitem{Kalayanamit2024}
{\sc P.~Kalayanamit}, {\em Sobolev regularity of the inverse for minimizers of
  the neo-{H}ookean energy satisfying condition {INV}}.
\newblock ArXiv:2405.12156, 2024.

\bibitem{Maly1995}
{\sc J.~Mal{\'y}}, {\em Examples of weak minimizers with continuous
  singularities}, Expo. Math., 13 (1995), pp.~446--454.

\bibitem{MarsdenHughes1983}
{\sc J.~E. Marsden and T.~J.~R. Hughes}, {\em Mathematical foundations of
  elasticity}.
\newblock Prentice-{Hall} {Civil} {Engineering} and {Engineering} {Mechanics}
  {Series}. {Englewood} {Cliffs}, {New} {Jersey}: {Prentice}-{Hall}, {Inc}.
  {XVIII}, 556 p. {\$} 58.00 (1983)., 1983.

\bibitem{Mueller1988}
{\sc S.~M{\"u}ller}, {\em Weak continuity of determinants and nonlinear
  elasticity. ({Continuit{\'e}} faible des d{\'e}terminants et applications
  {\`a} l'{\'e}lasticit{\'e} non lin{\'e}aire)}, C. R. Acad. Sci., Paris,
  S{\'e}r. I, 307 (1988), pp.~501--506.

\bibitem{MullerTangYan1994}
{\sc S.~M{\"u}ller, T.~Qi, and B.~S. Yan}, {\em On a new class of elastic
  deformations not allowing for cavitation}, Ann. Inst. Henri Poincar{\'e},
  Anal. Non Lin{\'e}aire, 11 (1994), pp.~217--243.

\bibitem{MullerSpector1995}
{\sc S.~M\"{u}ller and S.~J. Spector}, {\em An existence theory for nonlinear
  elasticity that allows for cavitation}, Arch. Rational Mech. Anal., 131
  (1995), pp.~1--66.

\bibitem{Rolfsen1976}
{\sc D.~Rolfsen}, {\em Knots and links}.
\newblock Mathematical {Lecture} {Series}. 7. {Berkeley}, {Ca}.: {Publish} or
  {Perish}, {Inc}., 1976.

\bibitem{Stein1993_HA_book}
{\sc E.~M. Stein}, {\em Harmonic analysis: {Real}-variable methods,
  orthogonality, and oscillatory integrals. {With} the assistance of {Timothy}
  {S}. {Murphy}}, vol.~43 of Princeton Math. Ser., Princeton, NJ: Princeton
  University Press, 1993.

\bibitem{Sverak1988}
{\sc V.~{\v{S}}ver{\'a}k}, {\em Regularity properties of deformations with
  finite energy}, Arch. Ration. Mech. Anal., 100 (1988), pp.~105--127.

\bibitem{VodopyanovGoldshtein1977}
{\sc S.~K. Vodop'yanov and V.~M. Gol'dshtein}, {\em Quasiconformal mappings and
  spaces of functions with generalized first derivatives}, Sib. Math. J., 17
  (1977), pp.~399--411.

\bibitem{Zeidler1986}
{\sc E.~Zeidler}, {\em Nonlinear functional analysis and its applications. {I}:
  {Fixed}-point theorems. {Transl}. from the {German} by {Peter} {R}.
  {Wadsack}}.
\newblock New {York} etc.: {Springer}-{Verlag}. {XXI}, 897 p. {DM} 298.00
  (1986)., 1986.

\end{thebibliography}
	
\end{document}